\documentclass[10pt]{amsart}
\usepackage{amsmath,amssymb,latexsym,soul,cite,color,enumitem,graphicx, tikz, mathtools, microtype,amsthm,fancyvrb}
\usepackage[left=2.65cm,right=2.65cm,top=3cm,bottom=3cm]{geometry}
\usepackage[english]{babel}
\usepackage[colorlinks=true,urlcolor=darkcerulean,citecolor=darkcerulean,linkcolor=darkcerulean,linktocpage,pdfpagelabels,bookmarksnumbered,bookmarksopen]{hyperref}
\definecolor{darkcerulean}{rgb}{0.03, 0.27, 0.49}

\numberwithin{equation}{section}
\newtheorem{theorem}{Theorem}[section]
\newtheorem{lemma}[theorem]{Lemma}
\newtheorem{proposition}[theorem]{Proposition}

\newtheorem{definition}[theorem]{Definition}
\theoremstyle{definition}
\newtheorem{remark}[theorem]{Remark}

\newcommand{\eps}{\varepsilon}
\newcommand{\fl}{(-\Delta)^{s\,}}
\newcommand{\flp}{(-\Delta)^{s} p}
\newcommand{\beq}{\begin{equation}}
\newcommand{\eeq}{\end{equation}}
\newcommand{\R}{{\mathbb R}}
\newcommand{\N}{{\mathbb N}}

\newcommand{\C}{{\mathbb C}}

\newenvironment{enumroman}{\begin{enumerate}

}{\end{enumerate}}

\title[Solutions of the fractional heat equation]{Existence and convexity of solutions\\ of the fractional heat equation}

\author[A.\ Greco]{Antonio Greco}
\author[A.\ Iannizzotto]{Antonio Iannizzotto}

\address{Dipartimento di Matematica e Informatica
\newline\indent
Universit\`a degli Studi di Cagliari
\newline\indent
Via Ospedale 72, 09124 Cagliari -- Viale L.\ Merello 92, 09123 Cagliari, Italy}
\email{greco@unica.it, antonio.iannizzotto@unica.it}


\begin{document}

\begin{abstract}
We prove that the initial-value problem for the fractional heat equation admits an entire solution provided that the (possibly unbounded) initial datum has a conveniently moderate growth at infinity. Under the same growth condition we also prove that the solution is unique. The result does not require any sign assumption, thus complementing the Widder's type theorem of Barrios et al.~\cite{BPSV} for positive solutions. Finally, we show that the fractional heat flow preserves convexity of the initial datum. Incidentally, several properties of stationary convex solutions are established.
\end{abstract}

\maketitle

\begin{center}
Version of \today.
\end{center}

\section{Introduction}\label{sec1}

\noindent
The fractional Laplacian operator is defined by
\beq\label{fl} \fl u(x)=C_{N,s}\,{\rm PV} \! \int_{\R^N}\frac{\, u(x)-u(y) \,}{\, |x-y|^{N+2s} \,}\,dx,
\quad s \in (0,1), \ N \ge 1,
\eeq
where
 $u \colon \R^N\to\R$ is a conveniently smooth function,
 `PV' stands for `principal value' and $C_{N,s}>0$ is a suitable normalizing constant (see Section~\ref{sec2} for a precise definition). Though being defined as an integral operator, $\fl$ can be seen as a fractional power of the Laplacian operator, and exhibits some similar properties and behavior to those of the Laplacian (for a detailed account on such properties, see~\cite{CS,CS1,C,CS2,DPV}). The outstanding feature of~$\fl$ is {\em non-locality}, which can be expressed as follows: the quantity $\fl u(x)$ depends not only on the values of $u$ in a neighborhood of $x$ (as is the case for the Laplacian), but rather on the values of $u$ at any point $y\in\R^N$.
\vskip2pt
\noindent
While convergence of the integral in~\eqref{fl} at $x$ depends on the regularity of the function $u$, convergence at infinity is obviously related to the asymptotic behavior of $u(y)$ as $|y|\to\infty$. In many references (see for instance~\cite{DPV}), therefore, $\fl u$ is only computed for $u$ lying in the Schwartz space~$\mathcal{S}$ of rapidly decaying $C^\infty$ functions, or in the fractional Sobolev space $H^s(\R^N)$ which also consists of functions decaying at infinity. That is also the reason why stationary problems involving $\fl$ on a bounded domain $\Omega\subset\R^N$ are often coupled with a Dirichlet-type exterior condition of the form $u(x)=0$ in $\R^N\setminus\Omega$ (see for instance~\cite{GS,IMS,RS,SV}, and~\cite{MN,SV1} on different ways to restrict fractional equations to bounded domains).
\vskip2pt
\noindent
Concerning evolutive problems, it has been pointed out that parabolic-type equations involving a classical derivative with respect to time and a fractional space diffusion are the natural outcome of certain stochastic processes (see~\cite{BGR,BV}, and~\cite{VTV} for a different approach based on the second law of thermodynamics). The simplest possible evolutive equation of such type is the {\em fractional heat equation}, and the main subject of this paper is the related initial-value problem:
\beq\label{ivp}
\begin{cases}
u_t+\fl u=0 & \text{in $\R^N\times(0,\infty)$,} \\
u(x,0)=u_0(x) & \text{in $\R^N$.}
\end{cases}
\eeq
The equation in~\eqref{ivp}, along with its non-linear variants, has been given considerable attention recently, both concerning regularity and existence/uniqueness results (see for instance~\cite{BPSV,CF,FR} and the website~\cite{mediawiki}). In particular, in~\cite{BPSV} a uniqueness result for the positive solution of~\eqref{ivp} was proved under the assumption that the initial datum~$u_0$ is positive (like in the classical result of~\cite{W} on the heat equation).
\vskip2pt
\noindent
Here we let $s$ span the whole interval $(0,1)$ and allow the initial datum~$u_0$ to be unbounded and sign-changing, subject to a polynomial bound of the type
\beq\label{gc0}
|u_0(x)|\le A_0+B_0 \, |x|^{2s-\sigma}
\quad
\mbox{for all $x \in \mathbb R^N$},
\eeq
with $A_0,B_0,\sigma>0$. We prove that problem~\eqref{ivp} admits exactly one (possibly sign-changing) solution $u(x,t)$ defined in $\mathbb R^N \times [0,+\infty)$ and satisfying
\beq\label{gc}
|u(x,t)|\le A(t)+B \, |x|^{2s-\sigma}
\quad \mbox{for all $(x,t) \in \R^N \times [0,+\infty)$}
\eeq
with $A$ being a sublinear function, i.e.,
\beq\label{att}
\lim_{t\to +\infty}\frac{\, A(t) \,}{t}=0,
\eeq
and~$B>0$. The proof is based on a weak maximum principle and some new estimates of the heat kernel, developing those of~\cite{BG,BGR,CKS}. We remark that space diffusion takes place in~$\R^N$ and the solution does not decay as $|x|\to\infty$, so the type of solutions we deal with are continuous functions satisfying the equation pointwise and are not required to belong to the standard fractional Sobolev classes (see Section~\ref{sec3} for details).
\vskip2pt
\noindent
The polynomial bounds in~\eqref{gc0} and~\eqref{gc} are natural for the following reasons. If, for a fixed~$t>0$, $u(x,t)$ grows as fast as $|x|^{2s}$ when $|x|\to\infty$, then the integral in~\eqref{fl} which defines $\fl u(x,t)$ may fail to converge. Furthermore, since the fractional heat kernel decays polynomially as $|x|\to +\infty$ (see Section~\ref{sec3}), the bound on~$u_0$ plays a role in the convergence of the convolution integral which defines the solution. In\cite{CS1} a solution of a stationary fractional equation is produced via the fractional heat kernel, and exhibits a growth of the type above. The uniqueness result in the present paper (claim~\ref{uniqueness} of Theorem~\ref{exun}) can be seen as a non-local analogue to the classical result for the heat equation (see~\cite[(1.36c)]{J}) ensuring uniqueness of the solution under the exponential growth condition
\[|u(x,t)|\le A \, e^\frac{\, B \, |x|^2 \,}{t} \ (A,B>0)\]
which is in turn related to the exponential decay of the classical heat kernel. The analogy with the classical case breaks down when it comes to \textit{existence}: indeed, Theorem~\ref{exun} \ref{solves} ensures existence of the canonical solution for all times $t \in [0,+\infty)$ provided that the initial datum~$u_0$ satisfies~\eqref{gc0} \textit{for some} $A_0,B_0 > 0$. By contrast, when $s = 1$ the canonical solution exists in the bounded interval $[0, \, \frac1{\, 4B \,})$ provided that
\begin{equation}\label{gc0exp}
|u_0(x)| \le A \, e^{B \, |x|^2}
\quad
\mbox{for all $x \in \mathbb R^N$}
\end{equation}
(see~\cite[(1.14)]{J}) hence in order to have existence for all $t \in [0,+\infty)$ it is required that \textit{for all} $B > 0$ there exists $A > 0$ such that \eqref{gc0exp} holds.
\vskip2pt
\noindent
We also address the study of geometric properties of the solution of problem~\eqref{ivp}, focusing on {\em convexity} in the space variable $x$ (see Section~\ref{sec4}). In the classical case~$s=1$ we know that the heat flow preserves convexity: quoting~\cite{IS}, ``it is easily seen that, when $\Omega=\R^N$, every solution of [the heat equation] with moderate growth at space infinity preserves the spatial concavity of [the initial datum] at any time~$t>0$''.
\vskip2pt
\noindent
We prove that a similar result holds in the fractional case $s \in (1/2, \, 1)$, and give precise statements in Theorem~\ref{geosol} and Theorem~\ref{geosolclassical}. In particular we prove that if the initial datum~$u_0$ is convex, then the solution $u(x,t)$ of problem~\eqref{ivp} not only is convex in~$x$ at any time~$t>0$, but also it is either strictly convex or its graph consists of parallel straight lines (see~\cite{G} for a similar result on stationary problems). To achieve the result we study very carefully the properties of the fractional Laplacian of convex functions in the stationary case. Once again, it is essential here that we do not consider space-decaying solutions, since a continuous function on~$\R^N$ with a decay at infinity cannot be convex. The condition $s \in (1/2, \, 1)$ is natural for two reasons. First, if $s\le 1/2$ the fractional Laplacian $\fl u$ does not converge for any convex function $u \colon \R^N\to\R$ (not even for affine functions). Second, the growth condition~\eqref{gc0} is consistent with convexity only if $2s>1$.
\vskip2pt
\noindent
The paper has the following structure: Section~\ref{sec2} deals with the definition and some properties of the stationary fractional Laplacian; Section~\ref{sec3} contains the mentioned existence and uniqueness result for the fractional heat equation; Section~\ref{sec4} is about fractional Laplacians of convex functions and the propagation of convexity through the fractional heat flow; and an Appendix (Section~\ref{sec5}) is devoted to proving some estimates on the fractional heat kernel and its derivatives.

\section{Some properties of the stationary fractional Laplacian}\label{sec2}

\noindent
First we need to make the definition in~\eqref{fl} more precise. To do so, we recall some well-known facts of measure ad integration theory. We will use the notation
\[\R^+=[0,+\infty),\quad \R^+_0=(0,+\infty), \quad \R^-=(-\infty,0], \quad \R^-_0=(-\infty,0).\]
Given a measurable (bounded or unbounded) set $D\subseteq\R^N$ and a measurable function $f\colon D\to\R$, we set $f^\pm(x)=\max\{\pm f(x),0\}$ for all $x\in D$ and consider the (Lebesgue) integrals
\[I^\pm=\int_D f^\pm(x)\,dx,\]
both taking values in $\R^+\cup\{+\infty\}$. If at least one of $I^\pm$ is finite, we say that $f$ is {\em integrable} in $D$, while if~both $I^\pm$ are finite we say that $f$ is {\em summable} in $D$. If $f$ is either integrable or summable we write
\[\int_D f(x)\,dx=I^+-I^-, \quad \int_D |f(x)|\,dx=I^++I^-.\]
Now let $s\in(0,1)$, $u\in C(\R^N)$, $x\in\R^N$ and define
\[C_{N,s}:=\frac{\, 2^{2s} s \, \Gamma(\frac{N}{2}+s) \,}{\pi^\frac{N}{2} \,
  \Gamma(1-s)}=\Big(\int_{\R^N}\frac{\, 1-\cos(x_1) \,}{|x|^{N+2s}}\,dx\Big)^{\!-1}\]
(see~\cite[Remark 3.11]{CS} and~\cite[formula (3.2)]{DPV}, respectively). For all $\eps>0$ we denote by $B_\eps(x)$ the open ball in~$\R^N$ centered at $x$ with radius~$\eps$ (for brevity we shall write $B_\eps=B_\eps(0)$), and consider the function
\begin{equation}\label{integrand}
y\mapsto\frac{\, u(x)-u(y) \,}{\, |x-y|^{N+2s} \,}, \quad y\in B^c_\eps(x)
\end{equation}
(notation: $D^c=\R^N\setminus D$). Now we define precisely~$\fl u$:

\begin{definition}\label{fldef}
Assume $s\in(0,1)$. If the function in~\eqref{integrand} is integrable for all $\eps>0$ and there exists
\[\lim_{\eps\to 0^+}\int_{B^c_\eps(x)}\frac{u(x)-u(y)}{\, |x-y|^{N+2s} \,}\,dy=: L\in\R\cup\{+\infty, -\infty\},\]
then we say that $\fl u$ is definite at~$x$. If, in addition, $L\in\R$, then we say that $\fl u$ converges at~$x$. In both cases we set
\[\fl u(x)=C_{N,s} \, L.\]
If the function in~\eqref{integrand} is not integrable for some $\eps_0>0$ (and hence for all $\eps\in(0,\eps_0)$), then we say that $\fl u$ is indefinite at~$x$.
\end{definition}

\noindent
If $u$ belongs to the Schwartz space~$\mathcal{S}$, an equivalent definition of~$\fl$ is given by
\beq\label{flf}
\fl u(x)=\mathcal{F}^{-1}\big(|\xi|^{2s} \, \mathcal{F}u(\xi)\big)(x),
\eeq
where $\mathcal{F}$ and~$\mathcal{F}^{-1}$ denote, respectively, the direct and inverse Fourier transforms, i.e.
\[\mathcal{F}v(\xi)=\int_{\R^N}e^{-ix\cdot\xi} \, v(x)\,dx, \quad \mathcal{F}^{-1}v(x)=\frac{1}{(2\pi)^N}\int_{\R^N}e^{ix\cdot\xi} \, v(\xi)\,d\xi.\]
It is also well-known that for all $u \in \mathcal{S}$ the following non-singular representation, which does not require the principal value, holds:
\beq\label{fl-dd}
\fl u(x)=\frac{\, C_{N,s} \,}{2}\int_{\R^N}\frac{\, 2u(x) - u(x + z) - u(x - z) \,}{|z|^{N+2s}}\,dz.
\eeq
The equivalence between \eqref{fl} and~\eqref{flf} via the normalization constant~$C_{N,s}$, together with a proof of~\eqref{fl-dd} is found, for instance, in~\cite{DPV}. 
We see next that such the representation~\eqref{fl-dd} also holds for less regular, possibly unbounded functions~$u$ satisfying a convenient growth condition. We also prove the continuity of~$\fl u(x)$:

\begin{theorem}\label{continuity}
Let $u\in C^2(\R^N)$ satisfy
\beq\label{t1}
|u(x)|\le A+B \, |x|^{2s-\sigma}
\ \text{for all $x\in\R^N$ ($A,B,\sigma>0$).}
\eeq
Then:
\begin{enumroman}
\item\label{continuity1} the fractional Laplacian $\fl u(x)$ converges for all $x\in\R^N$, and the representation~\eqref{fl-dd} holds;
\item\label{continuity2} the function $\fl u$ is continuous in~$\R^N$.
\end{enumroman}
\end{theorem}
\begin{proof}
Choose $x \in \R^N$. For every $r > \varepsilon > 0$ we have
\begin{align*}
\int_{B^c_\varepsilon(x)}\frac{\, u(x) - u(y) \,}{|x - y|^{N + 2s}}\,dy &= \int_{B_r \cap B^c_\varepsilon}\frac{\, u(x) - u(x+z) \,}{|z|^{N + 2s}}\,dz+\int_{B^c_r}\frac{\, u(x) - u(x+z) \,}{|z|^{N + 2s}}\,dz \\
&= \int_{B_r \cap B^c_\varepsilon}\frac{\,u(x)-u(x-z) \,}{|z|^{N + 2s}}\,dz+\int_{B^c_r}\frac{\, u(x) - u(x-z) \,}{|z|^{N + 2s}}\,dz,
\end{align*}
where all integrals are finite by~\eqref{t1}. By adding the two equalities, we obtain
\begin{align}\label{summation}
2\int_{B^c_\varepsilon(x)}\frac{\, u(x) - u(y) \,}{|x - y|^{N + 2s}}\,dy &= \int_{B_r\cap B^c_\varepsilon}\frac{\, 2u(x) - u(x + z) - u(x - z) \,}{|z|^{N + 2s}}\,dz \\
\nonumber&+ \int_{B^c_r}\frac{\, 2u(x) - u(x + z) - u(x - z) \,}{|z|^{N + 2s}}\,dz.
\end{align}
Besides, for all $z\in B_r\setminus \{0\} \supset B_r \cap B^c_\varepsilon$ we have by Taylor expansion
\begin{equation}\label{Taylor}
\frac{\, |2u(x)-u(x+z)-u(x-z)| \,}{|z|^{N+2s}}\le\frac{1}{\, |z|^{N+2s-2} \,}\,\sup_{\eta\in B_r(x)}|D^2 u(\eta)|.
\end{equation}
Since the right-hand side is summable on~$B_r$ (being $N+2s-2<N$), we deduce
\[\lim_{\varepsilon \to 0^+}\int_{B_r\cap B^c_\varepsilon}\frac{\, 2u(x)-u(x+z)-u(x-z) \,}{|z|^{N+2s}}\,dz=\int_{B_r}\frac{\, 2u(x)-u(x+z)-u(x-z) \,}{|z|^{N+2s}}\,dz.\]
Moreover, the last integral in~\eqref{summation} is finite by~\eqref{t1} and does not depend on~$\eps$. So, we pass to the limit in~\eqref{summation} as $\eps\to 0^+$ and we get
\[\frac{2}{C_{N,s}}\,\fl u(x)=\int_{\R^N}\frac{2u(x)-u(x+z)-u(x-z)}{|z|^{N+2s}}\,dz,\]
which proves~\ref{continuity1}.
\vskip2pt
\noindent
To prove~\ref{continuity2} we fix an arbitrary $x_0 \in \R^N$ and we set
\[M=\!\sup_{\eta\in B_1(x_0)}|D^2 u(\eta)|.\]
For all $x \in\R^N$ and all $r > 0$, by~\eqref{Taylor} we have
\begin{align}\label{first_integral}
\Big|\int_{B_r}\frac{\, 2u(x)-u(x+z)-u(x-z) \,}{|z|^{N+2s}}\,dz\Big| &\le N \omega_N\sup_{\eta\in B_r(x)}|D^2 u(\eta)|\,\int_0^r \frac{dr}{\, r^{2s - 1} \,} \\
\nonumber&= N \omega_N\,\frac{\, r^{2 - 2s} \,}{\, 2 - 2s \,}\,\sup_{\eta\in B_r(x)}|D^2 u(\eta)|.
\end{align}
Now, for any $\delta>0$ we choose $r_\delta \in (0,1/2)$ s.t.\ $N \omega_N \, M \,r_\delta^{2 - 2s} < (2 - 2s) \, \delta$. Since $r_\delta < 1/2$, for all $x
\in B_{r_\delta}(x_0)$ we have $B_{r_\delta}(x) \subset B_1(x_0)$ and therefore
\[\Big|\int_{B_{r_\delta}}\frac{\, 2u(x)-u(x-z)-u(x+z) \,}{|z|^{N+2s}}\,dz\Big|<\delta.\]
This and the representation~\eqref{fl-dd} imply that for every $x \in B_{r_\delta}(x_0)$ we have
\begin{align}\label{bound}
\int_{B^c_{r_\delta}}\frac{\, 2u(x) - u(x + z) - u(x - z) \,}{|z|^{N + 2s}}\,dz-\delta &< \frac{2}{C_{N,s}}\,\fl u(x) \\
\nonumber&< \int_{B^c_{r_\delta}}\frac{\, 2u(x) - u(x + z) - u(x - z) \,}{|z|^{N + 2s}}\,dz+\delta.
\end{align}
The estimate~\eqref{t1} ensures the existence of a summable majorant: indeed $u(x)$ stays bounded as $x$ ranges in~$B_{r_\delta}(x_0)$, while $|u(x \pm z)|\le A + B \, (|x_0|+r_\delta+|z|)^{2s - \sigma}$. Thus, we may pass to the limit under the sign of integral as $x \to x_0$ and get
\begin{align*}
& \int_{B^c_{r_\delta}}\frac{\, 2u(x_0) - u(x_0 + z) - u(x_0 - z) \,}{|z|^{N + 2s}}\,dz-\delta \le \frac2{\, C_{N,s} \,}\,\liminf_{x \to x_0}\,\fl u(x) \\
&\le \frac2{\, C_{N,s} \,}\,\limsup_{x \to x_0}\,\fl u(x) \le \int_{B^c_{r_\delta}}\frac{\, 2u(x_0) - u(x_0 + z) - u(x_0 - z) \,}{|z|^{N + 2s}}\,dz+\delta.
\end{align*}
Using the estimate~\eqref{bound} at $x = x_0$, the inequalities above become
\[\fl u(x_0) - C_{N,s} \, \delta \le \liminf_{x \to x_0}\,\fl u(x) \le \limsup_{x \to x_0}\,\fl u(x) \le \fl u(x_0)+C_{N,s} \, \delta.\]
Since $\delta$ is arbitrary, \ref{continuity2} follows.
\end{proof}

\noindent
We aim at finding sufficient conditions s.t.\ $\fl u(x)$ vanishes as $|x|\to +\infty$. We begin by establishing some estimates:

\begin{lemma}\label{estimate}
Let $u\in C^2(\R^N)$ satisfy
\beq\label{d2}
|D^2u(x)|\le \frac{C}{|x|^\alpha} \ \text{for all $x\in\R^N\setminus\{0\}$ ($C>0$, $\alpha>0$).}
\eeq
Moreover, if $\alpha \ge 1$ assume that there exist $A,B>0$ s.t.
\beq\label{esti1}
|u(x)|\le A+B \, |x|^{2-\alpha} \ \text{for all $x\in\R^N$.}
\eeq
Then there exist $A_1,B_1>0$ s.t.
\beq\label{esti2}
|2u(x)-u(x-z)-u(x+z)|\le A_1+B_1 \, |z|^{2-\alpha} \ \text{for all $x,z\in\R^N$.}
\eeq
\end{lemma}
\begin{proof}
First we recall that all norms are equivalent on the space of $N\times N$ matrices, in particular \eqref{d2} also holds for the norm
\[|D^2u(x)|=\max\big\{|\lambda|:\,\lambda \ \text{eigenvalue of $D^2u(x)$}\big\}\]
We show that when $\alpha \in (0,1)$ assumption~\eqref{d2} implies~\eqref{esti1}.
For all $x\in\R^N\setminus\{0\}$ we set $\xi=x/|x|$ and denote by $u_\xi$, $u_{\xi\xi}$ the first and second derivatives, respectively, of~$u$ along the direction~$\xi$. By~\eqref{d2} we have
\begin{align*}
|u_\xi(x)| &= \Big|u_\xi(0)+|x|\int_0^1u_{\xi\xi}(tx)\,dt\Big| \\
&\le |u_\xi(0)|+|x|\int_0^1\frac{C}{\, |tx|^\alpha \,}\,dt \\
&\le |D u(0)|+\frac{C}{\, 1-\alpha \,} \, |x|^{1-\alpha}.
\end{align*}
Similarly, we have
\begin{align*}
|u(x)| &= \Big|u(0)+|x|\int_0^1u_\xi(tx)\,dt\Big| \\
&\le |u(0)|+|x| \, |D u(0)|+|x|\int_0^1\frac{C}{\, 1-\alpha \,} \, |tx|^{1-\alpha}\,dt \\
&\le |u(0)|+|x| \, |D u(0)|+\frac{C}{\, 1-\alpha \,} \, \frac{\, |x|^{2-\alpha} \,}{2-\alpha}.
\end{align*}
By $2-\alpha>1$, taking $A,B>0$ big enough we have~\eqref{esti1}. So, in fact we can count on~\eqref{esti1} for all $\alpha>0$.
\vskip2pt
\noindent
Now we prove~\eqref{esti2}. For all $x,z\in\R^N$, $z\neq 0$ we define the segment
\[S:=\big\{x+tz:\,t\in[-1,1]\big\}.\]
By Taylor expansion, we have
\beq\label{l1}
|2u(x)-u(x+z)-u(x-z)|\le |z|^2 \max_{y\in S}|D^2 u(y)|.
\eeq
Let $y_0\in S$ be the unique point minimizing the distance from~$0$. We distinguish two cases:
\begin{itemize}[leftmargin=0.5cm]
\item if $|y_0|\ge |z| > 0$, then by assumption~\eqref{d2} we have
\[\max_{y\in S}|D^2 u(y)|\le\frac{C}{\, |y_0|^\alpha \,}.\]
This and~\eqref{l1} imply
\[|2u(x)-u(x+z)-u(x-z)|\le C \, |z|^{2-\alpha};\]
\item if $|y_0|<|z|$, then we have $|x| \le |x - y_0| +| y_0| < 2 \, |z|$ as well as $|x\pm z| < 3 \, |z|$, so by~\eqref{esti1} we have
\[|2u(x)-u(x+z)-u(x-z)|\le 4A+4B \, (3 \, |z|)^{2-\alpha}.\]
\end{itemize}
In both cases, for $A_1,B_1>0$ big enough (independently of~$x$) we get~\eqref{esti2}.
\end{proof}

\noindent
The next result highlights a condition to have $\fl u$ vanishing at infinity:

\begin{proposition}\label{vanish1}
Let $u\in C^2(\R^N)$ satisfy~\eqref{t1} with some $A,B,\sigma > 0$ as well as~\eqref{d2} with $\alpha =2 - 2s + \sigma > 0$. Then $\fl u(x)$ converges at all $x \in \R^N$, the representation~\eqref{fl-dd} holds, the function $\fl u$ is  continuous in~$\R^N$, and
\beq\label{vanish}
\lim_{|x|\to +\infty}\fl u(x)=0.
\eeq
\end{proposition}
\begin{proof}
In view of Theorem~\ref{continuity} we only have to prove~\eqref{vanish}. To this aim we fix $\delta > 0$ and we observe that assumption~\eqref{t1} may be rewritten as~\eqref{esti1}. Hence by Lemma~\ref{estimate} there exist constants $A_1,B_1$ s.t.\ \eqref{esti2} holds. Since $2s+\alpha-1 = 1 + \sigma >1$, we can find $r_\delta>0$ so large that
\[\int_{r_\delta}^{+\infty}\frac{\, A_1+B_1 \, r^{2-\alpha} \,}{r^{1+2s}}\,dr<\delta,\]
hence by~\eqref{esti2} we have
\beq\label{last_integral}
\Big|\int_{B_{r_\delta}^c}\frac{\, 2u(x)-u(x+z)-u(x-z) \,}{|z|^{N+2s}}\,dz\Big| \le N\omega_N\int_{r_\delta}^{+\infty}\frac{\, A_1+B_1 \, r^{2-\alpha} \,}{r^{1+2s}}\,dr < N\omega_N \, \delta.
\eeq
Besides, letting $r = r_\delta$ in~\eqref{first_integral}, we have
\beq\label{interior}
\Big|\int_{B_{r_\delta}}\frac{\, 2u(x)-u(x+z)-u(x-z) \,}{|z|^{N+2s}}\,dz\Big| \le N\omega_N \, \frac{r_\delta^{2-2s}}{\, 2-2s \,}\sup_{\eta\in B_{r_\delta}(x)}|D^2u(\eta)|.
\eeq
By assumption~\eqref{d2} we have, in particular,
$$
\lim_{|x| \to +\infty}
\Big(
\sup_{\eta\in B_{r_\delta}(x)}|D^2 u(\eta)|
\Big)
=
0
$$
and therefore there exists $R_\delta$ s.t.\ for all $x \in B^c_{R_\delta}$ we have
\[\sup_{\eta\in B_{r_\delta}(x)}|D^2 u(\eta)|<\frac\delta{\, r_\delta^{2 - 2s} \,}.\]
Using~\eqref{last_integral}, \eqref{interior}, and the representation~\eqref{fl-dd}, we have for all $x \in B^c_{R_\delta}$
\[\big|\fl u(x)\big| \le \frac{\, C_{N,s} \,}{2} \, N\omega_N \, \Big(1+\frac{1}{\, 2-2s \,}\Big) \, \delta.\]
Since $\delta>0$ is arbitrary, \eqref{vanish} follows.
\end{proof}

\section{The fractional heat equation}\label{sec3}

\noindent
For any function $u\colon\R^N\times\R^+_0\to\R$ and for each point $(x,t)\in\R^N\times\R^+_0$ we will denote by $\fl u(x,t)$ the fractional Laplacian of $u(\cdot,t)$ at~$x$. We will denote by $Du$ and $D^2u$ the gradient and Hessian matrix, respectively, of~$u$ with respect to the space variable~$x$, and by~$u_t$ the derivative with respect to the time variable~$t$.
\vskip2pt
\noindent
The fractional heat equation reads
\beq\label{fhe}
u_t+\fl u=0 \ \text{in $\R^N\times\R^+_0$.}
\eeq
The following definitions explain what we mean by a (super-, sub-) solution:

\begin{definition}\label{fhe-sol}
Let $u\in C(\R^N\times\R^+_0)$ be a function s.t.\ $u(x,\cdot)\in C^1(\R^+_0)$ for all $x\in\R^N$, and $\fl u$ converges at any $(x,t)\in\R^N\times\R^+_0$. We say that $u$ is
\begin{enumroman}
\item\label{fhe-sol1} a solution of equation~\eqref{fhe}, if for all $(x,t)\in\R^N\times\R^+_0$
\[u_t(x,t)+\fl u(x,t)=0;\]
\item\label{fhe-sol2} a (strict) super-solution of~\eqref{fhe}, if for all $(x,t)\in\R^N\times\R^+_0$
\[u_t(x,t)+\fl u(x,t)\ge0 \ (>0);\]
\item\label{fhe-sol3} a (strict) sub-solution of~\eqref{fhe}, if for all $(x,t)\in\R^N\times\R^+_0$
\[u_t(x,t)+\fl u(x,t)\le0 \ (<0).\]
\end{enumroman}
\end{definition}

\noindent
Further, given an initial datum $u_0\in C(\R^N)$, we will consider problem~\eqref{ivp}:

\begin{definition}\label{ivp-sol}
Let $u \in C(\R^N \times \R^+)$ be a solution of equation~\eqref{fhe} s.t.\ $u(x,0) = u_0(x)$ for all $x\in\R^N$. Then we say that $u$ is a solution of the initial-value problem~\eqref{ivp}. Super- and sub-solutions of~\eqref{ivp} are defined analogously.
\end{definition}

\noindent
The above notions are usually referred to as {\em strong} solutions, see for instance~\cite[Definition~1.3]{BPSV}. Note that solutions of equation~\eqref{fhe} are only defined for~$t>0$, while solutions of problem~\eqref{ivp} are defined for $t\ge 0$.
\vskip2pt
\noindent
A major tool in the study of~\eqref{fhe} is the {\em fractional heat kernel}, defined by means of the Fourier integral
\beq\label{fhk}
p(x,t)=\frac{1}{(2\pi)^N}\int_{\R^N}e^{ix\cdot\xi-t|\xi|^{2s}}\,d\xi
\eeq
for all $(x,t)\in\R^N\times\R^+_0$. From past and recent literature we know that $p\in C^\infty(\R^N\times\R^+_0)$, takes on positive values, satisfies for all $t\in\R^+_0$
\beq\label{volume}
\int_{\R^N}p(x,t)\,dx=1,
\eeq
and is a solution of equation~\eqref{fhe} ($p$ can be seen as a solution of the initial-value problem~\eqref{ivp} with $u_0$ replaced by the Dirac delta centered at~$0$). Moreover, $p$ satisfies the following two-sided estimate for all $(x,t)\in\R^N\times\R^+_0$:
\beq\label{ke0}
C_1 \, \min\Big\{\frac{1}{\, t^\frac{N}{2s} \,},\,\frac{t}{\, |x|^{N+2s} \,}\Big\}\le p(x,t)\le
C_2 \, \min\Big\{\frac{1}{\, t^\frac{N}{2s} \,},\,\frac{t}{\, |x|^{N+2s} \,}\Big\} \quad
(0<C_1<C_2),
\eeq
see~\cite{BPSV}, \cite[Eq.\ (4)]{BGR}, \cite[Eq.\ (1.1)]{CKS}. It is understood that $\min\{a,\frac{\, t \,}0\} = a$ for all $a \in \mathbb R$ and $t \in \mathbb R^+_0$, so \eqref{ke0} makes sense also at~$x = 0$. The same convention holds in~\eqref{dps} and~\eqref{dpt} below. In the next proposition we state an estimate from above of the spatial derivatives of~$p(x,t)$ of any order, as well as of the first-order time derivative. For any multi-index $\alpha=(\alpha_1,\ldots, \alpha_N)$ of height $|\alpha|= \alpha_1 + \dots + \alpha_N = k\ge 1$ we set, as usual,
\[D^\alpha p(x,t)=\frac{\partial^k}{\partial x_1^{\alpha_1}\ldots \partial x_N^{\alpha_N}}\,p(x,t),\]
while $p_t(x,t)$ denotes the derivative of $p$ with respect to $t$.

\begin{proposition}\label{dp}
The fractional heat kernel $p$ defined in~\eqref{fhk} satisfies the following estimates:
\begin{enumroman}
\item\label{dp1} for all integer $N,k\ge 1$ there exists $C_{N,k}>0$ s.t.\ for all multi-index $\alpha$ with $|\alpha|=k$ and all $(x,t)\in\R^N\times\R^+_0$
\beq\label{dps}
|D^\alpha p(x,t)|\le C_{N,k}\,\min\Big\{\frac1{t^\frac{N + k}{2s}},\,\frac{t}{|x|^{N+2s+k}}\Big\};
\eeq
\item\label{dp2} there exists $C>0$ s.t.\ for all $(x,t)\in\R^N\times\R^+_0$
\beq\label{dpt}
|p_t(x,t)| \le C \, \min\Big\{\frac1{t^{\frac{N}{2s} + 1} \,}, \, \frac1{|x|^{N + 2s}}\Big\}.
\eeq
\end{enumroman}
\end{proposition}

\noindent
Similar estimates are found, for instance, in \cite[(2.2) and Lemma 2.1]{IKM}. To make the paper self-contained, a proof of Proposition~\ref{dp} is given in the Appendix.
\vskip2pt
\noindent
As said in the Introduction, estimate~\eqref{ke0} yields for $p(\cdot,t)$ a polynomial decay as $|x|\to +\infty$, which leads us to consider solutions of problem~\eqref{ivp} satisfying the growth condition~\eqref{gc},
where $B>0$ and $\sigma \in (0,2s)$ are constants, and $A\in C(\R^+)$ is a positive function s.t.~\eqref{att} holds. We will see that such a solution exists, provided the initial datum $u_0$ satisfies an analogous growth condition for all $x\in\R^N$, namely~\eqref{gc0} with $A_0,B_0>0$ and $\sigma \in (0,2s)$. First we prove a technical lemma on sub-solutions:

\begin{lemma}\label{max}
Let $u$ be a sub-solution of~\eqref{fhe}, attaining its maximum at $(x_0,t_0)\in\R^N\times\R^+_0$. Then $u(\cdot,t_0)$ is constant in~$\R^N$.
\end{lemma}
\begin{proof}
Since $u(x_0,\cdot)$ has a maximum at $t_0$, we have $u_t(x_0,t_0)=0$. Since $u$ is a sub-solution of~\eqref{fhe}, we have $\fl u(x_0,t_0)\le 0$. Set for all $\eps>0$
\[I_\eps=\int_{B^c_\eps(x_0)}\frac{\, u(x_0,t_0)-u(y,t_0) \,}{|x_0-y|^{N+2s}}\,dy\le 0\]
(note that the integrand is non-negative in~$\R^N$). Then, $\eps\mapsto I_\eps$ is a non-negative, non-increasing mapping, while
\[\lim_{\eps\to 0^+}I_\eps=\frac{\fl u(x_0,t_0)}{C_{N,s}}\le 0.\]
Thus we have $I_\eps=0$ for all $\eps>0$, which in turn implies $u(y,t_0)=u(x_0,t_0)$ for all $y\in\R^N$.
\end{proof}

\noindent
Now we prove a weak maximum principle for sub-solutions of problem~\eqref{ivp}, under a one-sided version of condition~\eqref{gc}:

\begin{theorem}\label{mp}
Let $u$ be a sub-solution of problem~\eqref{ivp} s.t.\ for all $(x,t)\in\R^N\times\R^+$
\[u(x,t)\le A(t)+B \, |x|^{2s-\sigma}\] where $B>0$ and $\sigma \in (0,2s)$ are
constants, and $A\in C(\R^+)$ is a positive function satisfying~\eqref{att}.
Then
\[\sup_{(x,t)\in\R^N\times\R^+}u(x,t)=\sup_{x\in\R^N}u_0(x).\]
\end{theorem}
\begin{proof}
Avoiding trivial cases, we may assume that $u_0$ is bounded from above. We argue by contradiction: assume that there exists $(x_1,t_1)\in\R^N\times\R^+_0$ s.t.
\beq\label{mpabs}
u(x_1,t_1)>\sup_{x\in\R^N}u_0(x)+\eps_1 \ (\eps_1>0).
\eeq
Let $v_0\in C^2(\R^N)$ be a positive function s.t.\ $v_0(x)=|x|^{2s-\sigma_0}$ for all $x\in B^c_1$ and some $\sigma_0\in(0,\sigma)$. By direct computation we have for all $x\in\R^N \setminus \{\, 0 \,\}$
\[|D^2v_0(x)|\le\frac{C}{\, |x|^{2-2s+\sigma_0} \,} \ (C>0).\]
Proposition~\ref{vanish1} implies that $\fl v_0\in C(\R^N)$ and $\fl v_0(x)\to 0$ as $|x|\to+\infty$. Let us pick $M>0$ s.t.
\[\inf_{x\in\R^N}\fl v_0(x)>-M.\]
Now set for all $(x,t)\in\R^N\times\R^+$
\[v(x,t)=v_0(x)+Mt.\]
It is easily seen that $v\in C(\R^N\times\R^+)$ is a positive strict super-solution of problem~\eqref{ivp} with initial datum~$v_0$. Set $\mu=\eps_1/v(x_1,t_1)>0$ and for all $(x,t)\in\R^N\times\R^+$
\[w(x,t)=u(x,t)-\mu \, v(x,t).\]
Then $w\in C(\R^N\times\R^+)$ is a strict sub-solution of problem~\eqref{ivp} with initial datum $u_0-\mu \, v_0$. Moreover, for all $(x,t)\in\R^N\times\R^+$ we have
\[w(x,t)\le(A(t)-\mu Mt)+(B \, |x|^{2s-\sigma}-\mu \, v_0(x)),\]
and the latter tends to $-\infty$ as soon as either $|x| \to +\infty$ or $t \to +\infty$ (recall that \eqref{att} holds, and that $2s-\sigma<2s-\sigma_0$). So, we can find $(x_0,t_0)\in\R^N\times\R^+$ s.t.
\[w(x_0,t_0)=\sup_{(x,t)\in\R^N\times\R^+}w(x,t).\]
More precisely we must have $t_0 > 0$ by~\eqref{mpabs} and because
\[w(x_0,t_0)\ge w(x_1,t_1)=u(x_1,t_1)-\eps_1>\sup_{x\in\R^N}u(x,0)\ge\sup_{x\in\R^N}w(x,0).\]
But then $w(\cdot,t_0)$ is constant by Lemma~\ref{max}, and therefore $\fl w(x_0,t_0)=0$. Furthermore $w_t(x_0,t_0)=0$ because such point is a maximizer, thus
\[w_t(x_0,t_0)+\fl w(x_0,t_0)=0,\]
which contradicts the fact that $w$ is a strict sub-solution. This concludes the proof.
\end{proof}

\noindent
The canonical solution of the initial-value problem~\eqref{ivp} is defined by means of the fractional heat kernel~$p$: given an initial datum $u_0\in C(\R^N)$ satisfying~\eqref{gc0}, we set for all $(x,t)\in\R^N\times\R^+_0$
\beq\label{cs}
u(x,t)=\int_{\R^N}u_0(y)\,p(x-y, \, t)\,dy,
\eeq
while we set $u(x,0)=u_0(x)$ for all $x\in\R^N$. The following is the main result of the paper, in which we prove that $u$ is actually a solution of~\eqref{ivp} and that it is unique under condition~\eqref{gc}:

\begin{theorem}\label{exun}
Let $u_0\in C(\R^N)$ satisfy the bound~\eqref{gc0}, and let $u$ be defined by~\eqref{cs}. Then:
\begin{enumroman}
\item\label{obeys} $u$ obeys to the bound~\eqref{gc} with a convenient function $A(t)$ satisfying~\eqref{att}, and some constant $B$;
\item\label{solves} $u$ is a solution of the initial-value problem~\eqref{ivp};
\item\label{uniqueness} if $u_1,u_2$ are solutions of~\eqref{ivp} corresponding to the same initial value~$u_0$ and s.t.\
\[|u_i(x,t)|\le A_i(t)+B_i \, |x|^{2s-\sigma} \quad \mbox{for $(x,t) \in \R^N \times \R^+$ and $i = 1,2$}\]
with two functions $A_1,A_2$ satisfying~\eqref{att} and some constants $B_1,B_2$, then $u_1 = u_2$.
\end{enumroman}
\end{theorem}
\begin{proof}
Preliminarily, we establish some bounds on the integrand in~\eqref{cs}. Fix $(x,t)\in\R^N\times\R^+_0$. Letting $R_t := t^\frac1{2s}$ we have
\beq\label{alternative}
\min\Big\{\frac{1}{\,t^\frac{N}{2s}\,},\,\frac{t}{\, |x-y|^{N+2s} \,}\Big\} =\begin{cases}
\displaystyle\frac{1}{\, t^\frac{N}{2s} \,} & \text{if $|x - y| < R_t$,} \\
\noalign{\medskip}
\displaystyle\frac{t}{\, |x-y|^{N+2s} \,} & \text{if $|x - y| \ge R_t$.}
\end{cases}
\eeq
Consequently, from assumption~\eqref{gc0} and estimate~\eqref{ke0} we have for all $x,y\in\R^N$ and all $t\in\R^+_0$
\beq\label{consequence}
|u_0(y) \, p(x-y, \, t)|\le C \, (1+|y|^{2s-\sigma})\,\Big(\frac1{\, t^\frac{N}{2s} \,} \, \chi_{B_{R_t}(x)}(y)+\frac{t}{\, |x-y|^{N+2s} \,} \, \chi_{B^c_{R_t}(x)}(y)\Big),
\eeq
with a convenient constant $C > 0$. Hence there exists a constant $C_0>0$ s.t.\ for all $(x,t) \in B_1(x_0) \times (t_0/2, \, t_0 + 1)$
\[|u_0(y) \, p(x-y, \, t)|\le\frac{C_0}{\, 1+|y|^{N+\sigma} \,}\]
from which we see that the convolution integral in~\eqref{cs} converges. So, $u(x,t)$ is well defined and continuous in $\R^N\times\R^+_0$. Using in a similar manner the estimate~\eqref{dpt} established for the time-derivative $p_t(x,t)$, we can also differentiate with respect to~$t$ in~\eqref{cs} and ensure that $u(x,\cdot)\in C^1(\R^+_0)$ for every $x \in\R^N$ with
\beq\label{ut}
u_t(x,t)=\int_{\R^N}u_0(y)\,p_t(x-y, \, t)\,dy.
\eeq
Now we prove~\ref{obeys}. We shall use the following elementary inequalities, holding for all $a,b\ge 0$:
\beq\label{elementary}
(a+b)^{2s - \sigma} \le
\begin{cases}
a^{2s - \sigma} +b^{2s - \sigma} &\mbox{if $2s - \sigma \in (0,1]$}\\
2^{2s - \sigma - 1} \, (a^{2s - \sigma} + b^{2s - \sigma}) &\mbox{if $2s - \sigma \in (1,2)$.}
\end{cases}
\eeq
From now on we denote by~$C>0$ a constant whose value may change from line to line. Integrating~\eqref{consequence} in~$dy$ over~$\R^N$ we find for all $x\in\R^N$ and $t\in\R^+_0$
\beq\label{half-done}
|u(x,t)| \le C \, \omega_N \, (1+(|x| + R_t)^{2s-\sigma})+ Ct \int_{B^c_{R_t}(x)}\frac{\, 1 + |y|^{2s - \sigma} \,}{\, |x - y|^{N + 2s} \,} \, dy,
\eeq
where $\omega_N = |B_1|$ and $R_t$ is defined as above. Hence we may focus on the last integral.  Since $|y| \le |x| + |y - x|$, using~\eqref{elementary} we have
\begin{align*}
\int_{B^c_{R_t}(x)}\frac{\, 1 + |y|^{2s - \sigma} \,}{\, |x - y|^{N + 2s} \,} \, dy &\le \int_{B^c_{R_t}(x)}\frac{\, 1 + (|x| + |y - x|)^{2s - \sigma} \,}{\, |x - y|^{N + 2s} \,} \, dy \\
&\le (1 + C \, |x|^{2s - \sigma}) \int_{B^c_{R_t}(x)} \frac{dy}{\, |x - y |^{N + 2s} \,} + C \int_{B^c_{R_t}(x)} \frac{dy}{\, |x - y |^{N + \sigma} \,} \\
&\le C \, (1 + |x|^{2s - \sigma}) \, \frac1{\, t \,} + C \, \frac1{\, t^\frac\sigma{2s} \,}.
\end{align*}
By plugging this estimate into~\eqref{half-done} we see that the prospective solution $u(x,t)$ satisfies the bound~\eqref{gc} for $t > 0$ with a convenient function $A(t)$ and some constant~$B$. Furthermore, since $u_0$ satisfies~\eqref{gc0} by assumption, we may replace $A(t)$ with $A_0 + A(t)$, and enlarge the constant~$B$ if necessary, so that the bound~\eqref{gc} holds for all $t \ge 0$, as claimed. We also get the sharper information that $A(t)$ not only satisfies~\eqref{att} but also $A(t) = O(t^{1 - \frac\sigma{2s}})$.
\vskip2pt
\noindent
Now we turn to~\ref{solves}. First, we need to prove that $u$ defined by~\eqref{cs} solves the fractional heat equation~\eqref{fhe} at any $(x,t)\in\R^N\times\R^+_0$. In what follows, we will omit dependence on both $x$ and $t$, which are considered fixed, so we will denote by $C=C_{x,t}>0$ a constant whose value may change from line to line. By~\eqref{ut} and recalling that $p$ solves~\eqref{fhe}, we see that
\beq\label{sol1}
u_t(x,t)=-\int_{\R^N}u_0(y) \, \flp(x-y,t)\,dy.
\eeq
To go further, it is essential to recall that $u(\cdot,t)$ satisfies the bound~\eqref{t1} and therefore, by Theorem~\ref{continuity}, the fractional Laplacian $\fl u(x,t)$ admits the representation~\eqref{fl-dd}. Clearly, by simple variable changes we have
\[u(x,t)=\int_{\R^N} u_0(x - y) \, p(y,t) \, dy,\]
\[u(x \pm z, \, t) = \int_{\R^N} u_0(x - y) \, p(y \pm z,t)\,dy \quad \mbox{for all $z\in\R^N$.}\]
By plugging these into the representation~\eqref{fl-dd} we arrive at an expression of $\fl u(x,t)$ by means of an iterated integral:
\begin{align}\label{iterated}
\fl u(x,t) &= \frac{\, C_{N,s} \,}2 \int_{\R^N}\frac{2u(x,t)-u(x+z, \, t)-u(x-z, \, t)}{|z|^{N+2s}}\,dz \\
\nonumber&= \frac{\, C_{N,s} \,}2 \int_{\R^N}\left(\int_{\R^N}u_0(x - y) \,\big(2p(y,t) - p(y + z, \, t) - p(y - z, \, t)\big) \, dy\right)\frac{dz}{\, |z|^{N + 2s} \,}.
\end{align}
In order to change the order of integration, we need to construct a majorant of
\[\varphi(y,z)=\Big|u_0(x - y) \,\frac{\, 2p(y,t) - p(y + z, \, t) - p(y - z, \, t) \,}{\, |z|^{N + 2s} \,}\Big|\]
summable in $\R^{2N}$. The first factor $u_0(x-y)$ is simply estimated by virtue of~\eqref{gc}. So we fix $y\in\R^N$ and we mainly consider the second factor, which is subject to the estimate~\eqref{dps} from Proposition~\ref{dp}:
\begin{itemize}[leftmargin=0.5cm]
\item for all $z\in B_1$ we have by Taylor expansion and~\eqref{dps} (with $|\alpha|=2$)
\begin{align*}
\varphi(y,z) &\le \frac{\, C \, (1+|y|^{2s-\sigma}) \,}{\,|z|^{N+2s-2}\,}\sup_{\eta\in B_1(y)}|D^2 p(\eta,t)| \\
&\le \frac{\, C \, (1+|y|^{2s-\sigma}) \,}{\,|z|^{N+2s-2} \, (1+|y|^{N+2s+2})\,};
\end{align*}
\item for all $z\in B_1^c$ we have by~\eqref{ke0} and~\eqref{alternative} (with $R_t=t^\frac{1}{2s}$ as above)
\[|p(y,t)| \le \frac{C}{\, 1 + |y|^{N + 2s} \,},\]
\[|p(y\pm z, \, t)|\le
\begin{cases}
C &\mbox{if $|y \pm z| < R_t$,} \\
\displaystyle\frac{C}{\, |y \pm z|^{N + 2s} \,} &\mbox{if $|y \pm z| \ge R_t$.}
\end{cases}\]
Thus, from
\begin{align*}
\varphi(y,z) \le C \, (1+|y|^{2s-\sigma}) \, \Big(\frac{\, 2 \, |p(y,t)| \,}{\,|z|^{N+2s}\,}+\frac{\, |p(y+z, \, t)| \,}{\,|z|^{N+2s}\,}+\frac{\, |p(y-z, \, t)| \,}{\,|z|^{N+2s}\,}\Big)
\end{align*}
we get
\begin{align*}
\varphi(y,z) &\le \frac{C \, (1+|y|^{2s-\sigma})}{\,|z|^{N+2s} \, (1+|y|^{N+2s})\,} \\
&+ \frac{C \, (1+|y|^{2s-\sigma}) \,}{\,|z|^{N+2s}\,} \, \chi_{B_{R_t}(-z)}(y) + \frac{\, C \, (1+|y|^{2s-\sigma}) \,}{\,|z|^{N+2s} \, |y+z|^{N+2s}\,} \, \chi_{B^c_{R_t}(-z)}(y) \\
&+ \frac{\, C \, (1+|y|^{2s-\sigma})\,}{\,|z|^{N+2s}\,} \, \chi_{B_{R_t}(z)}(y) + \frac{C \, (1+|y|^{2s-\sigma})}{\,|z|^{N+2s} \, |y-z|^{N+2s}\,} \, \chi_{B^c_{R_t}(z)}(y).
\end{align*}
\end{itemize}
In summary, $\varphi(y,z)$ satisfies for all $y,z\in\R^N$
\begin{align*}
\varphi(y,z) &\le \frac{\, C \, (1+|y|^{2s-\sigma}) \,}{\,|z|^{N+2s-2} \, (1+|y|^{N+2s+2})\,} \, \chi_{B_1}(z) + \frac{C \, (1+|y|^{2s-\sigma})}{\,|z|^{N+2s} \, (1+|y|^{N+2s})} \, \chi_{B_1^c}(z) \\
&+ \frac{\, C \, (1+|y|^{2s-\sigma}) \,}{|z|^{N+2s}} \, \chi_{B_1^c}(z) \, \chi_{B_{R_t}(-z)}(y) + \frac{C \, (1+|y|^{2s-\sigma})}{\, |z|^{N+2s} \, |y+z|^{N+2s} \,} \, \chi_{B^c_1}(z) \, \chi_{B_{R_t}^c(-z)}(y) \\
&+ \frac{\, C \, (1+|y|^{2s-\sigma}) \,}{|z|^{N+2s}} \, \chi_{B_1^c}(z) \, \chi_{B_{R_t}(z)}(y) + \frac{C \, (1+|y|^{2s-\sigma})}{\, |z|^{N+2s} \, |y-z|^{N+2s} \,} \, \chi_{B^c_1}(z) \, \chi_{B_{R_t}^c(z)}(y).
\end{align*}
All terms in the right-hand side above are summable in $\R^{2N}$. The least immediate part regards the terms involving $|y\pm z|^{N+2s}$, but using the change of variable $w=y\pm z$ and the inequalities~\eqref{elementary} we easily get
\begin{align*}
\int_{B_1^c}\frac1{\, |z|^{N+2s} \,}\int_{B^c_{R_t}(\mp z)}\frac{\, 1+|y|^{2s-\sigma} \,}{\, |y\pm z|^{N+2s} \,}\,dy\,dz &= \int_{B_1^c}\frac1{\, |z|^{N+2s} \,}\int_{B^c_{R_t}}\frac{\, 1+|w\mp z|^{2s-\sigma} \,}{|w|^{N+2s}}\,dw\,dz \\
&\le \int_{B_1^c} C \, \frac{\, 1+|z|^{2s-\sigma} \,}{|z|^{N+2s}}\,dz < +\infty.
\end{align*}
Thus, we may conclude that $\varphi(y,z)$ is summable in $\R^{2N}$ and change the order of integration in~\eqref{iterated}, getting
\begin{align*}
\fl u(x,t) &= \frac{\, C_{N,s} \,}{2}\int_{\R^N}u_0(x-y) \, \Big(\int_{\R^N}\frac{\, 2p(y,t)-p(y+z)-p(y-z) \,}{|z|^{N+2s}}\,dz\Big)\,dy \\
&= \int_{\R^N}u_0(x-y)\,\fl p(y,t)\,dy,
\end{align*}
where in the end we have used the representation~\eqref{fl-dd} for the fractional Laplacian of $p(\cdot,t)$ (which is legitimate by Theorem~\ref{continuity} and the regularity and bounds of $p(\cdot,t)$). Comparing with~\eqref{sol1}, we see that $u$ solves~\eqref{fhe}, as claimed.
\vskip2pt
\noindent
To complete the proof of~\ref{solves}, we still need to check that $u(x,t)$ is
continuous at $t = 0$. Let $(x_n)$, $(t_n)$ be sequences in~$\R^N$, $\R^+_0$ respectively, s.t.\ $x_n\to x_0$ and $t_n\to 0$. We claim that
\beq\label{sol4}
\lim_n u(x_n,t_n)=u_0(x_0).
\eeq
Fix $\eps>0$. There exists $r>0$ s.t.\ $|u_0(y)-u_0(x_0)|<\eps$ for all $y\in
B_r(x_0)$. For $n\in\N$ large enough we have $x_n\in B_{r/2}(x_0)$ and
$t_n^\frac{1}{2s}<\frac{r}{2}$. Furthermore, for all $n\in\N$ we have by~\eqref{volume}
\[\int_{\R^N}p(x_n - y, \, t_n)\,dy=1,\]
so by the definition~\eqref{cs} of $u(x,t)$ and the estimate~\eqref{ke0} of the
kernel~$p$ we can compute
\begin{align*}
|u(x_n,t_n)-u_0(x_0)| &= \Big|\int_{\R^N}(u_0(y)-u_0(x_0)) \, p(x_n-y, \, t_n)\,dy\Big| \\
&\le \int_{B_r(x_0)} |u_0(y)-u_0(x_0)| \, p(x_n-y, \, t_n)\,dy+\int_{B_r^c(x_0)} |u_0(y)-u_0(x_0)| \, p(x_n-y, \, t_n)\,dy \\
&\le \eps+ C_2 \, t_n \! \int_{B^c_r(x_0)}
\frac{\, |u_0(y)-u_0(x_0)| \,}{\, |x_n-y|^{N+2s} \,}\,dy,
\end{align*}
where $C_2$ is the constant in~\eqref{ke0}. We focus on the latter integral. For all $y\in B^c_r(x_0)$ and $n\in\N$ large enough we have $|x_n-y|\ge\frac{r}{2}$. This and assumption~\eqref{gc0} imply
\[\frac{\, |u_0(y)-u_0(x_0)| \,}{|x_n-y|^{N+2s}}\le C \, \frac{\,1+|y|^{2s-\sigma} \,}{\, (|y-x_0|-(r/2))^{N+2s} \,},\]
and the latter is summable in $B^c_r(x_0)$. Recalling that $t_n \to 0$, for $n\in\N$ large enough we may write
\[|u(x_n,t_n)-u_0(x_0)|<2\eps,\]
which proves~\eqref{sol4}. Thus, \ref{solves} is achieved.
\vskip2pt
\noindent
Finally, we prove~\ref{uniqueness}. The difference $w = u_1 - u_2$ is a solution of problem~\eqref{ivp} with initial value $w(x,0) = 0$ in~$\R^N$. Since $w$ satisfies the estimate~\eqref{gc} with $A(t) = A_1(t) + A_2(t)$ and $B = B_1 +B_2$, by Theorem~\ref{mp} we have
\[\sup_{(x,t)\in\R^N\times\R^+}(u_1(x,t) - u_2(x,t))=0.\]
By reverting the roles of $u_1$ and~$u_2$ we conclude that $u_1 = u_2$ in $\R^N \times \R^+$ and conclude the proof.
\end{proof}

\section{Convex solutions}\label{sec4}

\noindent
This section is devoted to the relation between convexity and the fractional Laplacian. We begin with some observations on the stationary case, and then we study how convexity propagates from the initial datum through the fractional heat flow. We will see how the study of convex functions leads naturally to considering the operator $\fl$ with $s\in(1/2,1)$.
\vskip2pt
\noindent
Obviously, if $u$ is a constant function, then $\fl u(x)=0$ for all $s\in(0,1)$ and all $x\in\R^N$. As soon as we consider non-constant, affine functions, a dichotomy appears:

\begin{lemma}\label{fl-affine}
Let $u\in C(\R^N)$ be a non-constant, affine function and $s\in(0,1)$:
\begin{enumroman}
\item\label{fl-affine1} if $s\le 1/2$, then $\fl u(x)$ is indefinite at any $x\in\R^N$;
\item\label{fl-affine2} if $s>1/2$, then $\fl u(x)=0$ at any $x\in\R^N$.
\end{enumroman}
\end{lemma}
\begin{proof}
To prove~\ref{fl-affine1}, we fix $x\in\R^N$, $\varepsilon > 0$ and show that the function in~(\ref{integrand}) is not integrable in~$B^c_\varepsilon(x)$. By assumption there exists $\xi\in\R^N\setminus\{0\}$ s.t.\ for all $y\in\R^N$
\[u(y)=u(x)+\xi\cdot (y - x).\]
Define
\begin{equation}\label{Bpm}
(B^c_\varepsilon(x))^\pm:=\big\{y\in B_\varepsilon^c(x) : \, \pm \xi \cdot (y - x) \le 0\big\},
\end{equation}
so that $B^c_\varepsilon(x)= (B^c_\varepsilon(x))^+ \cup (B^c_\varepsilon(x))^-$. We have
$$
\int_{(B^c_\varepsilon(x))^\pm}\frac{u(x)-u(y)}{\, |x-y|^{N+2s} \,}\,dy
=
\int_{(B^c_\varepsilon(x))^\pm}\frac{\xi \cdot (x - y)}{\, |x-y|^{N+2s} \,}\,dy
.
$$
Passing to polar coordinates by $y - x = \rho \, w$, where $\rho \ge 0$ and $|w| = 1$, the volume element~$dy$ becomes $\rho^{N - 1} \, d\rho \, d{\mathcal H}^{N - 1}$ and the set $(B^c_\varepsilon(x))^\pm$ is transformed into the Cartesian product $[\varepsilon,+\infty) \times \partial B_1^\pm$. Here $\partial B_1^\pm = \{\, w \in \partial B_1 : \pm \xi \cdot w \le 0 \,\}$ and $d{\mathcal H}^{N - 1}$ denotes the $(N - 1)$-dimensional Hausdorff measure on the unit spherical surface $\mathbb S^{N - 1} = \partial B_1$. Hence
\begin{align*}
\pm\int_{(B^c_\varepsilon(x))^\pm}\frac{\xi \cdot (x - y)}{\, |x-y|^{N+2s} \,}\,dy
&=
-
\int_\varepsilon^{+\infty} \frac{d\rho}{\, \rho^{2s} \,}
\int_{\partial B_1^\pm} (\pm\xi \cdot w) \, d{\mathcal H}^{N - 1}
\end{align*}
Since $2s\le 1$, and by the definition of~$\partial B_1^\pm$, we have
$$
\int_\varepsilon^{+\infty} \frac{d\rho}{\, \rho^{2s} \,} = +\infty
\qquad\mbox{and}\qquad
\int_{\partial B_1^\pm} (\pm\xi \cdot w) \, d{\mathcal H}^{N - 1} \in \mathbb R_0^-
,
$$
therefore
\begin{equation}\label{affineball}
\pm\int_{(B^c_\varepsilon(x))^\pm}\frac{\xi \cdot (x - y)}{\, |x-y|^{N+2s} \,}\,dy = +\infty,
\end{equation}
so the function in~(\ref{integrand}) is not integrable in~$B^c_\varepsilon(x)$, as claimed. By Definition~\ref{fldef}, $\fl u$ is indefinite at $x$.
\vskip2pt
\noindent
We prove~\ref{fl-affine2}. We assume $s>1/2$, fix $x\in\R^N$ and $\eps>0$, and note that for all $y\neq x$
\[\Big|\frac{\, u(x)-u(y) \,}{\, |x-y|^{N+2s} \,}\Big|\le\frac{|\xi|}{\, |x-y|^{N+2s-1} \,},\]
and the latter is summable in $B^c_\eps(x)$ as $N+2s-1>N$. Furthermore, by symmetry we have
\[\int_{B^c_\eps(x)}\frac{\, u(x)-u(y) \,}{\, |x-y|^{N+2s} \,}\,dy=0,\]
so clearly $\fl u(x)=0$.
\end{proof}

\noindent
If we consider convex functions, the dichotomy still takes place. Moreover, in the case $s>1/2$, the fractional Laplacian admits the non-singular representation~\eqref{fl-dd}, in which the principal value can be omitted: this is mainly due to the fact that, for a convex function $u$, we have for all $x,z\in\R^N$
\[2u(x)-u(x+z)-u(x-z)\le 0.\]

\begin{lemma}\label{fl-convex}
Let $u \colon \R^N\to\R$ be a non-affine, convex function and $s\in(0,1)$:
\begin{enumroman}
\item\label{fl-convex1} if $s\le 1/2$, then for any $x\in\R^N$ the fractional Laplacian $\fl u(x)$ is either indefinite or $-\infty$;
\item\label{fl-convex2} if $s>1/2$, then $\fl u(x)\in\R^-_0\cup\{-\infty\}$ at each $x\in\R^N$, and\/~\eqref{fl-dd} holds.
\end{enumroman}
\end{lemma}
\begin{proof}
First we prove~\ref{fl-convex1}. Let $s\le 1/2$ and fix $x \in \R^N$, $\varepsilon > 0$. Since $u$ is convex and non-constant, there exist $a \in \R$ and $\xi\in\R^N \setminus \{0\}$ s.t.\ the inequality $u(y) \ge a + \xi \cdot y$ holds for all $y\in\R^N$, and therefore
\[u(x) - u(y)\le u(x) - a - \xi \cdot y =C + \xi \cdot (x - y),\]
where $C = u(x) - a - \xi \cdot x$. Integrating over the set $(B^c_\varepsilon(x))^-$ defined as in~\eqref{Bpm}, and taking~\eqref{affineball} into account, we get
\[\int_{(B^c_\varepsilon(x))^-}\frac{\, u(x) - u(y) \,}{\, |x - y|^{N + 2s} \,} \, dy \le \int_{(B^c_\varepsilon(x))^-}\frac{C}{\, |x - y|^{N + 2s} \,} \, dy+\int_{(B^c_\varepsilon(x))^-}\frac{\, \xi \cdot (x - y) \,}{\, |x - y|^{N + 2s} \,} \, dy = -\infty,\]
because the integral of $C/|x - y|^{N + 2s}$ over $(B^c_\varepsilon(x))^-$ is finite. Consequently, the integral
\[\int_{B^c_\varepsilon(x)}\frac{\, u(x) - u(y) \,}{\, |x - y|^{N + 2s} \,} \, dy\]
is either indefinite or~$-\infty$, and the same conclusion holds for $\fl u(x)$, as claimed.
\vskip2pt
\noindent
We prove~\ref{fl-convex2}. Assume $s>1/2$ and fix $x\in\R^N$. By convexity, there exists $\xi\in\R^N$ (possibly 0) s.t.\ for all $y\in\R^N$
\[u(y)\ge u(x)+\xi\cdot(y-x).\]
Set for all $y\in\R^N$
\[v(y):=u(y)-(u(x)+\xi\cdot(y-x)),\]
so that $v\in C(\R^N)$ is convex, non-negative in~$\R^N$, not identically zero, and $v(x)=0$. Thus we can define
\[I:=\int_{\R^N}\frac{v(y)}{\, |x-y|^{N+2s} \,}\,dy, \quad I_\eps:=\int_{B^c_\eps(x)}\frac{v(y)}{\, |x-y|^{N+2s} \,}\,dy \quad (\eps>0).\]
Clearly $I,I_\eps\in\R^+_0\cup\{+\infty\}$, the mapping $\eps\mapsto I_\eps$ is non-increasing, and $I_\eps\to I$ as $\eps\to 0^+$. For all $\eps>0$ we have, by $N+2s>N+1$ and symmetry,
\[\int_{B^c_\eps(x)}\frac{u(x)-u(y)}{\, |x-y|^{N+2s} \,}\,dy=-I_\eps+\int_{B^c_\eps(x)}\frac{\xi\cdot(x-y)}{\, |x-y|^{N+2s} \,}\,dy=-I_\eps.\]
Thus, by Definition~\ref{fldef} we have
\beq\label{flv}
\fl u(x)=-C_{N,s}\,I\in\R^-_0\cup\{-\infty\}.
\eeq
To conclude, we perform the change of variable $y = x \pm z$ in the
integral~$I$, thus getting
\[\int_{\R^N} \frac{\, v(x \pm z) \,}{|z|^{N + 2s}} \, dz=I.\]
By summation of the two equalities we obtain
\beq\label{i1}
\frac1{\, 2 \,}
\int_{\R^N}\frac{\, v(x + z) + v(x - z) \,}{|z|^{N+2s}}\,dz=I.
\eeq
Combining~\eqref{flv}, \eqref{i1}, and the definition of $v$, we have
\[\fl u(x)=-\frac{\, C_{N,s} \,}2 \int_{\R^N}\frac{\, v(x + z) + v(x - z) \,}{|z|^{N+2s}}\,dz=\frac{\, C_{N,s} \,}{2}\int_{\R^N}\frac{\, 2u(x) - u(x + z) - u(x - z) \,}{|z|^{N+2s}}\,dz,\]
which proves~\eqref{fl-dd} and concludes the argument.
\end{proof}

\begin{remark}
Both cases in~\ref{fl-convex1} of Lemma~\ref{fl-convex} may occur. For instance, fix $N=1$, $s\in(0,1/2]$, and for all $x\in\R$ set
\[u(x)=\begin{cases}
x & \text{if $x\ge 0$} \\
\lambda x & \text{if $x<0$,}
\end{cases}\]
where $\lambda\in(-\infty,1)$. An easy computation shows that $\fl u(0)=-\infty$ if $\lambda\le 0$, while it is indefinite if $\lambda\in(0,1)$, thus complying with either the first or the second case of~\ref{fl-convex1}.
\end{remark}

\noindent
The analogue of Theorem~\ref{continuity} for convex functions only requires a one-sided growth condition:

\begin{proposition}\label{fl-convgro}
Let $s\in(1/2,1)$, $u\in C^2(\R^N)$ be a convex function satisfying
\beq\label{gce}
u(x)\le A+B \, |x|^{2s-\sigma} \ \text{for all $x\in\R^N$ ($A,B,\sigma>0$).}
\eeq
Then $\fl u(x)$ converges for all $x\in\R^N$ and $\fl u\in C(\R^N)$.
\end{proposition}
\begin{proof}
It is not restrictive to assume $\sigma\in(0,2s-1)$. We note that, since $u$ is
convex and $2s-\sigma>1$, an estimate analogous to~\eqref{gce} holds in fact
also from below: i.e., taking $A,B>0$ bigger if necessary, for all $x\in\R^N$ we
have~\eqref{t1}. Hence the conclusion follows from Theorem~\ref{continuity}.
\end{proof}

\begin{remark}
From Lemma~\ref{fl-convex} and Proposition~\ref{fl-convgro} we see that when dealing with a convex function~$u$, it is quite natural to restrict our attention to the case $s\in(1/2,1)$ as we pointed out in the Introduction: first, because otherwise $\fl u$ does not converge; and second, because condition~\eqref{gce} is consistent with convexity only if $2s>1$.
\end{remark}

\noindent
The next result is an alternative to Proposition~\ref{vanish1} to prove that the fractional Laplacian $\fl u$ vanishes at infinity. Here, a milder decay condition than~\eqref{d2} is used together with the convexity of~$u$ and a further condition of geometric type (which excludes affine functions). We include such result for completeness, as it will not be used afterwards:

\begin{proposition}\label{vanish2}
Let $s\in(1/2,1)$ and let $u\in C^2(\R^N)$ be a convex function satisfying~\eqref{gce} together with
\beq\label{d2-alt}
\lim_{|x|\to +\infty}|D^2u(x)|=0,
\eeq
\beq\label{geo}
2 \, Du(x)\neq D u(x+z)+D u(x-z) \ \text{for all $x,z\in B^c_R$ ($R>0$).}
\eeq
Then $\fl u(x)$ converges at every $x \in \R^N$, and~\eqref{vanish} holds.
\end{proposition}
\begin{proof}
By Proposition~\ref{fl-convgro}, $\fl u(x)$ converges at every $x\in\R$. By Lemma~\ref{fl-convex}, the representation~\eqref{fl-dd} holds for any $x\in\R^N$, and for every $r >R$ we may write
\beq\label{split}
\frac2{C_{N,s}} \, \fl u(x)=\int_{B_r} \frac{\, 2u(x) - u(x + z) - u(x - z) \,}{|z|^{N + 2s}} \, dz+\int_{B^c_r} \frac{\, 2u(x) - u(x + z) - u(x - z) \,}{|z|^{N + 2s}} \, dz.
\eeq
Set for all $(x,z)\in\R^{2N}$
\[\psi(x,z):=2u(x)-u(x+z)-u(x-z).\]
Then $\psi\in C^2(\R^{2N})$ is non-positive and by Taylor expansion and~\eqref{d2-alt} we have for all $z\in\partial B_r$
\[\lim_{|x|\to\infty}\psi(x,z)=0.\]
So, we can find $(x_r,z_r)\in\R^N\times\partial B_r$ which minimizes~$\psi$ in the same set. By Lagrange's rule there exists $\lambda\in\R$ s.t.
\[\begin{cases}
D_x\psi(x_r,z_r)=2Du(x_r)-Du(x_r+z_r)-Du(x_r-z_r)=0 \\
D_z\psi(x_r,z_r)=-Du(x_r+z_r)+Du(x_r-z_r)=\lambda z_r.
\end{cases}\]
By the first equality above and~\eqref{geo} we have $|x_r|<R$, hence $|x_r\pm z_r|<R+r$. Summarizing, we have for all $(x,z)\in\R^{2N}$ with $|z|=r\ge R$
\[\psi(x,z)\ge\psi(x_r,z_r)\ge 2\inf_{\eta\in B_R}u(\eta)-u(x_r+z_r)-u(x_r-z_r)\ge -A_1-B_1 \, |z|^{2s-\sigma}\]
for convenient $A_1,B_1>0$, where in the end we have used~\eqref{gce} along with $|x_r\pm z_r| < R+|z|$. So, for all $\delta>0$ we can find $r_\delta\ge R$ s.t.\ for all $x\in\R^N$
\[0\ge \int_{B^c_{r_\delta}}\frac{\, 2u(x)-u(x+z)-u(x-z) \,}{|z|^{N+2s}}\,dz> -\delta.\]
Besides, letting $r = r_\delta$ in estimate~\eqref{first_integral} we get
\[0\ge \int_{B_{r_\delta}}\frac{\, 2u(x)-u(x+z)-u(x-z) \,}{|z|^{N+2s}}\,dz \ge -N\omega_N \, \frac{r_\delta^{2-2s}}{\, 2-2s \,}\sup_{\eta\in B_{r_\delta(x)}}|D^2u(\eta)|,\]
and by~\eqref{d2-alt} the latter can be estimated from below by $-\delta$ provided $|x|$ is large enough. Thus, from~\eqref{split} we have for all $x\in\R^N$ with $|x|$ large enough
\[0\ge\fl u(x)\ge - C_{N,s} \,\delta,\]
and \eqref{vanish} follows.
\end{proof}

\noindent
We can now prove a detailed result about convexity of solutions of problem~\eqref{ivp} with a convex initial datum. First, for all $v\in C(\R^N)$, $\xi\in\R^N\setminus\{0\}$, we say that $v$ is $\xi$-{\em ruled} if the graph of~$v$ consists of straight lines $\ell \subset \mathbb R^{N + 1}$ whose projections onto~$\mathbb R^N$ have the direction~$\xi$, i.e., for all $x\in\R^N$ and $\mu\in\R$
\[2v(x)-v(x+\mu \, \xi)-v(x-\mu \, \xi)=0.\]
Note that if a $\xi$-ruled function~$v$ is convex, then any two lines $\ell_1,\ell_2$ in the graph of~$v$ whose projections onto~$\mathbb R^N$ have the direction~$\xi$ are parallel to each other. We also recall an elementary fact about convex functions: let $v\in C(\R^N)$ be convex and $S\subset\R^{N+1}$ be a line segment s.t.\ the graph of~$v$ contains the endpoints of~$S$ and at least one further point of~$S$, then $S$ is contained in the graph of~$v$.
\vskip2pt
\noindent
We shall prove that, if the initial datum~$u_0$ is convex, the solution $u(x,t)$ of problem~\eqref{ivp} stays convex in~$x$ at any time~$t$, is non-decreasing in~$t$ at any point~$x$ (in fact, it is strictly increasing except in a trivial case), and furthermore it is either strictly convex or ruled in some direction, at any time~$t$ (and such feature does not change in time):

\begin{theorem}\label{geosol}
Let $s\in(1/2,1)$, $u_0\in C(\R^N)$ be a convex function satisfying for all $x\in\R^N$
\beq\label{geogro}
u_0(x)\le A_0+B_0 \, |x|^{2s-\sigma} \ (A_0,B_0,\sigma>0),
\eeq
and $u$ be the solution of problem~\eqref{ivp} defined by~\eqref{cs}. Then:
\begin{enumroman}
\item\label{convexity} $u(\cdot,t)$ is convex for all $t\in\R^+_0$;
\item\label{ruled} if $u_0$ is $\xi$-ruled for some $\xi\in\R^N\setminus\{0\}$, then $u(\cdot,t)$ is $\xi$-ruled for all $t\in\R^+_0$, otherwise $u(\cdot,t)$ is strictly convex for all $t\in\R^+_0$.
\item\label{monotonicity} $u_t(x,t)\ge 0$ for all $(x,t)\in\R^N\times\R^+_0$, moreover $u_t(x_0,t_0)=0$ at some $(x_0,t_0)\in\R^N\times\R^+_0$ iff $u_0$ is affine and $u(x,t)=u_0(x)$ for all $(x,t)\in\R^N\times\R^+_0$.
\end{enumroman}
\end{theorem}
\begin{proof}
Since $2s>1$ and $u_0$ is convex, we may assume $\sigma\in(0,2s-1)$ and that the bound~\eqref{geogro} holds from below as well, i.e., that \eqref{gc0} holds. By Theorem~\ref{exun}, $u$ is the only solution of problem~\eqref{ivp} satisfying~\eqref{gc}.
\vskip2pt
\noindent
First we prove~\ref{convexity}. Fix $t\in\R^+_0$. By definition~\eqref{cs} and convexity of~$u_0$, and recalling that $p(x,t)>0$ for all $(x,t)\in\R^N\times\R^+_0$ (see~\eqref{ke0}), we have for all $x,y\in\R^N$
\begin{align*}
&2u(x,t)-u(x+y, \, t)-u(x-y, \, t) \\
&= 2\int_{\R^N}u_0(z) \, p(x-z, \, t)\,dz-\int_{\R^N}u_0(z) \, p(x+y-z, \, t)\,dz-\int_{\R^N}u_0(z) \, p(x-y-z, \, t)\,dz \\
&= 2\int_{\R^N}u_0(z) \, p(x-z, \, t)\,dz-\int_{\R^N}u_0(z'+y) \, p(x-z', \, t)\,dz'-\int_{\R^N}u_0(z''-y) \, p(x-z'', \, t)\,dz'' \\
&= \int_{\R^N}\big(2u_0(z)-u_0(z+y)-u_0(z-y)\big) \, p(x-z, \, t)\,dz \le 0,
\end{align*}
where we have set $z'=z-y$, $z''=z+y$, and then we have denoted $z$ all integration variables. So $u(\cdot,t)$ is convex (incidentally, the same argument proves that $u(\cdot,t)$ is strictly convex if $u_0$ is such).
\vskip2pt
\noindent
Now we prove~\ref{ruled}. Assume that $u_0$ is $\xi$-ruled for some $\xi\in\R^N\setminus\{0\}$. Then, arguing as above, we have for all $(x,t)\in\R^N\times\R^+_0$, $\mu>0$
\[2u(x,t)-u(x+\mu \, \xi,t)-u(x-\mu \, \xi,t)=\int_{\R^N}\big(2u_0(z)-u_0(z+\mu \, \xi)-u_0(z-\mu \, \xi)\big) \, p(x-z, \, t)\,dz=0,\]
i.e., $u(\cdot,t)$ is $\xi$-ruled.
\vskip2pt
\noindent
To complete the proof of~\ref{ruled} we show that if for some $t\in\R^+_0$ the function $u(\cdot,t)$ is not strictly convex, then the initial datum~$u_0$ is $\xi$-ruled for some $\xi\in\R^N$. Indeed, if $u(\cdot,t)$ is not strictly convex, then by the elementary property recalled above its graph contains a segment
\[S = \big\{(x,u(x,t)) \in \mathbb R^{N + 1}:\,x = \overline x \pm \mu \, \xi,\ \mu \in [0,1]\big\}\]
for some $\overline x\in\R^N$ and $\xi\in\R^N\setminus\{0\}$. Hence we have for all $\mu\in[0,1]$
\[0=2u(\overline x,t)-u(\overline x+\mu \, \xi,t)-u(\overline x-\mu \, \xi,t)=\int_{\R^N}\big(2u_0(z)-u_0(z+\mu \, \xi)-u_0(z-\mu \, \xi)\big) \, p(\overline x-z, \, t)\,dz.\]
By positivity of $p$ and convexity of~$u_0$ we deduce that for all $x\in\R^N$, $\mu\in[0,1]$
\beq\label{line}
2u_0(x)-u_0(x+\mu \, \xi)-u_0(x-\mu \, \xi)=0.
\eeq
Shifting $x$ to $x\pm\xi$ and applying an iterative argument shows that \eqref{line} holds in fact for all $\mu\ge 0$, i.e., $u_0$~is $\xi$-ruled, and the proof of~\ref{ruled} is complete.
\vskip2pt
\noindent
Now we prove~\ref{monotonicity}. Fix $(x,t)\in\R^N\times\R^+_0$. From part~\ref{convexity} and Theorem~\ref{exun} we know that $u(\cdot,t)\in C(\R^N)$ is a convex function satisfying~\eqref{gce} for convenient $A,B>0$, so Lemmas~\ref{fl-affine}, \ref{fl-convex} yield $\fl u(x,t)\in\R^-$. By equation~\eqref{fhe} we have then
\beq\label{incr}
u_t(x,t)=-\fl u(x,t)\ge 0.
\eeq
Assume now that $u_t(x_0,t_0)=0$ at some $(x_0,t_0)\in\R^N\times\R^+_0$. Then, by~\eqref{fhe} we have $\fl u(x_0,t_0)=0$, which by part~\ref{convexity} and Lemma~\ref{fl-convex} implies that $u(\cdot,t_0)$ is affine. For all $(x,t)\in\R^N\times(0,t_0)$, by~\eqref{incr} we have
\[u(x,t)\le u(x,t_0),\]
i.e., $u(\cdot,t)$ is a convex function defined in~$\R^N$ and bounded from above by the affine function $u(\cdot,t_0)$. Thus, $u(\cdot,t)$ is itself affine and the difference $u(\cdot,t) - u(\cdot,t_0)$ is constant. Passing to the limit as $t\to 0^+$, we deduce that~$u_0$ is affine. But then by definition~\eqref{cs} we have  $u(x,t)=u_0(x)$ for all $(x,t)\in\R^N\times\R^+_0$.
Now the proof is concluded.
\end{proof}

\noindent
Let us put into evidence that a similar result (apart from the maximal existence time) also holds in the case when $s = 1$, i.e.\ for the Laplacian. More precisely, recall that the classical heat kernel $K(x - y, \, t)$ is given by
$$
K(x - y, \, t)
=
\frac1{\, (4\pi \, t)^{N/2} \,} \, e^{-\frac{\, |x - y|^2 \,}{4t}}
.
$$
It is well known (see~\cite{J})  that if the initial datum $u_0$ satisfies~\eqref{gc0exp} then we may let
$$
T = \frac1{\, 4B \,}
$$
and define $u(x,0) = u_0(x)$ for all $x \in \mathbb R^N$ and
\begin{equation}\label{csclassical}
u(x,t)
=
\int_{\R^N} u_0(y) \, K(x-y, \, t)\,dy
\quad
\mbox{for all $(x,t) \in \mathbb R^N \times (0,T)$,}
\end{equation}
thus obtaining the canonical solution of the initial-value problem
\begin{equation}\label{classicalivp}
\begin{cases}
u_t = \Delta u & \text{in $\R^N\!\times\mathbb R^+_0$,} \\
u(x,0)=u_0(x) & \text{in $\R^N$.}
\end{cases}
\end{equation}
We have:

\begin{theorem}\label{geosolclassical}
Let $u_0\in C(\R^N)$ be a convex function satisfying for all $x\in\R^N$
\beq\label{expgro}
u_0(x) \le A \, e^{B \, |x|^2}
\quad
(A,B > 0)
.
\eeq
Then the three claims of Theorem~\ref{geosol}, with every instance of the half-line~$\mathbb R^+_0$ replaced by the bounded interval $(0,T)$, hold for the canonical solution $u(x,t)$ of problem~{\rm(\ref{classicalivp})}.
\end{theorem}

\begin{proof}
As in Theorem~\ref{geosol}, since $u_0$ is convex we may assume that the bound~\eqref{expgro} is satisfied both from above and from below, i.e., \eqref{gc0exp} holds. The proof of claims \ref{convexity} and \ref{ruled} follows the same lines as there.
\vskip2pt
\noindent
To prove \ref{monotonicity}, observe that $u(x,t)$ is convex in~$x$ for all $t \ge 0$ by \ref{convexity}, hence $u_t(x,t) = \Delta u(x,t) \ge 0$ for all $t \in (0,T)$. Furthermore, if the initial datum~$u_0$ is affine then the canonical solution is $u(x,t) = u_0(x)$ and therefore $u_t$ vanishes identically. To complete the proof, suppose that $u_t(x_0,t_0) = 0$ at some $(x_0,t_0) \in \mathbb R^N \times (0,T)$. Recall that $u(x,t)$ is smooth for $t > 0$, hence letting $v(x,t) = u_t(x,t)$ we have that $v$ is a non-negative solution of the heat equation
\[v_t = \Delta v \ \text{in $\mathbb R^N \times (0,T)$.}\]
Since we are assuming $v(x_0,t_0) = 0$, by the strong minimum principle for the heat equation (see~\cite[Chapter~7, Theorem~11]{Evans} or~\cite[Chapter~3, Theorem~5]{Protter-Weinberger}) we get $v(x,t) = 0$ for all $(x,t) \in \mathbb R^N \times (0,t_0]$, and the argument proceeds as in the proof of Theorem~\eqref{geosol}. But then $u(x,t)$ is a spatially convex function which is also harmonic (in~$x$) for all $t \in (0,t_0]$, hence $u(x,t)$ is affine in~$x$ for all $t \in (0,t_0]$. Since $u_0(x) \le u(x,t)$ for each $t \in (0,T)$, and $u_0$ is convex by assumption, it must be also affine. Now using~\eqref{csclassical} we obtain $u(x,t) = u_0(x)$ for all $t \in (0,T)$ (in fact, for all $t \in \mathbb R_0^+$) and the proof is complete.
\end{proof}

\section{Appendix: Estimates on the fractional heat kernel}\label{sec5}

\noindent
This section is devoted to studying the asymptotic behavior of the fractional heat kernel $p(x,t)$ defined in~\eqref{fhk}, along with its partial derivatives of any order with respect to the space variable(s) $x_1, \dots, x_N$, $N \ge 1$. We will see that such functions have a polynomial decay as $|x|\to\infty$, and prove Proposition~\ref{dp} as a special case (see~\cite{IKM,KR} and the references therein for similar estimates).
\vskip2pt
\noindent
We begin by recalling that $p(\cdot,t)$ is radially symmetric for all $t\in\R^+_0$ and satisfies the following scaling property: for all $(x,t)\in\R^N\times\R^+_0$ we have
\beq\label{scale}
p(x,t)=t^{-\frac{N}{2s}} \, p(t^{-\frac{1}{2s}} \, x, \, 1)
\eeq
as it can be easily seen by changing variable in~\eqref{fhk}. From~\cite[Eq.\ (2.1)]{BG} (where $\alpha$ stands for our $2s$ and $f_\alpha(t,x)$ for our $p(x,t)$) and from~\eqref{scale} we have for all $(x,t)\in\R^N\times\R^+_0$
\beq\label{rad}
p(x,t)=\frac{t^{-\frac{N}{2s}}}{\, (2\pi)^\frac{N}{2} \,} \, F_N(t^{-\frac1{2s}} \, |x|),
\eeq
where we have set for all $r\in\R^+_0$
\[F_N(r)=r^\frac{2-N}{2}\int_0^\infty e^{-\rho^{2s}}\rho^\frac{N}{2}J_\frac{N-2}{2}(r\rho)\,d\rho.\]
Here $J_\nu\in C^\infty(\C)$ is the Bessel function with order $\nu\in\R$. We know that $J_\nu(z)$ has a polynomial decay as $|z|\to\infty$ and satisfies the following identity for all $z\in\C$, $\nu\in\R$:
\beq\label{bes}
zJ'_\nu(z)-\nu J_\nu(z)=-zJ_{\nu+1}(z).
\eeq
So, we are reduced to the study of~$F_N$. Such a study goes back to~\cite{P} for $N=1$. By the decay of $J_\nu$, the integral which defines~$F_N$ converges for all $r\in\R^+_0$. Moreover we have $F_N\in C^\infty(\R^+_0)$ and we can differentiate with respect to $r$ within the integral.
\vskip2pt
\noindent
One special feature of~$F_N$ is that its derivatives $D^k F_N$ of any order~$k$ can be expressed by means of functions $F_{N + 2j}$ ($j = 1, \dots, k$):

\begin{lemma}\label{fn1}
For all integer $k\ge 1$ there exist non-negative constants $\alpha_{j,k}\in\R^+$, defined for $j = 0, \dots, k + 1$ and positive iff $k \le 2j \le 2k$, s.t.\ for all $N\ge 1$ and all $r\in\R^+_0$
\beq\label{derf}
D^k F_N(r)=\sum_{k \le 2j \le 2k} (-1)^j \, \alpha_{j,k} \, r^{2j-k} \, F_{N+2j}(r).
\eeq
\end{lemma}
\begin{proof}
We argue by induction on $k$. First, setting $k=1$, we have for all $N \ge 1$
and $r \in \mathbb R^+_0$
\begin{align}\label{basefn}
DF_N(r) &= \frac{\, 2-N \,}{2} \, r^{-\frac{N}{2}}\int_0^\infty e^{-\rho^{2s}}\rho^\frac{N}{2}J_\frac{N-2}{2}(r\rho)\,d\rho+r^\frac{2-N}{2}\int_0^\infty e^{-\rho^{2s}}\rho^\frac{N+2}{2}J'_\frac{N-2}{2}(r\rho)\,d\rho \\
&= r^{-\frac{N}{2}}\int_0^\infty e^{-\rho^{2s}}\rho^\frac{N}{2}\Big(\frac{\, 2-N \,}{2} \, J_\frac{N-2}{2}(r\rho)+r\rho\, J'_\frac{N-2}{2}(r\rho)\Big)\,d\rho \nonumber \\
&= r^{-\frac{N}{2}}\int_0^\infty e^{-\rho^{2s}}\rho^\frac{N}{2}\big(-r\rho \, J_\frac{N}{2}(r\rho)\big)\,d\rho
= -r \, F_{N+2}(r), \nonumber
\end{align}
where we have used~\eqref{bes}. So \eqref{derf} holds for $k = 1$ with $\alpha_{1,1} = 1$ and $\alpha_{0,1} = \alpha_{2,1} = 0$.
\vskip2pt
\noindent
Now assume that \eqref{derf} holds true for some $k\ge 1$. By such assumption and~\eqref{basefn} we have for all $N\ge 1$ and all $r\in\R^+_0$
\begin{align}\label{sums}
D^{k+1}F_N(r) &= D\Big[\sum_{k \le 2j \le 2k} (-1)^j \, \alpha_{j,k} \, r^{2j-k} \, F_{N+2j}(r)\Big] \\
&= \sum_{k \le 2j \le 2k} (-1)^j \, (2j-k) \, \alpha_{j,k}\, r^{2j-k-1} \, F_{N+2j}(r)-\!\!\sum_{k \le 2j \le 2k} (-1)^j \, \alpha_{j,k}\, r^{2j-k+1} \, F_{N+2j+2}(r). \nonumber\\
\noalign{\medskip}
&=: S_1+S_2.\nonumber
\end{align}
Since the coefficient $2j - k$ vanishes when $k$ is even and $j=k/2$, and since $\alpha_{k + 1, \, k} = 0$ we may rewrite the first sum in~\eqref{sums} as follows:
\[S_1=\sum_{k + 1 \le 2j \le 2(k + 1)} (-1)^j \, (2j-k) \, \alpha_{j,k}\, r^{2j-k-1} \, F_{N+2j}(r).\]
Letting $h = j + 1$, and recalling that $\alpha_{h - 1, k} = 0$ when $k + 1 =2h$, the last sum in~\eqref{sums} becomes
\begin{align*}
S_2 &= \sum_{k \le 2(h - 1) \le 2k} (-1)^h \, \alpha_{h - 1,k} \, r^{2h-k-1} \, F_{N+2h}(r) \\
&= \sum_{k + 1 \le 2h \le 2(k + 1)} (-1)^h \, \alpha_{h - 1,k} \, r^{2h-k-1} \, F_{N+2h}(r).
\end{align*}
Hence \eqref{sums} is reduced to
\[D^{k+1}F_N(r) = \sum_{k + 1 \le 2j \le 2(k + 1)} (-1)^j \alpha_{j,k+1}\, r^{2j-(k+1)} F_{N+2j}(r),\]
with $\alpha_{j, \, k+1} = (2j - k) \, \alpha_{j,k} + \alpha_{j - 1, \, k}$. So \eqref{derf} also holds for $k+1$, which completes the proof.
\end{proof}

\noindent
From Lemma~\ref{fn1} we deduce the asymptotic behavior of all derivatives of~$F_N$ as $r\to\infty$:

\begin{lemma}\label{fn2}
For all integer $N\ge 1$, $k\ge 0$ the following limit is finite:
\beq\label{lim}
\lim_{r \to \infty} r^{N+2s+k} \, D^k F_N(r)=\ell_{N,k},
\eeq
with
\[\ell_{N,0}=\frac{\, 2^\frac{N+4s}{2}s \,}{\pi}\,\sin(\pi s)\,\Gamma\Big(\frac{N+2s}{2}\Big)\,\Gamma(s),\]
\begin{equation}\label{ell_Nk}
\ell_{N,k}=(-1)^k \, \ell_{N,0} \prod_{j = 0}^{k - 1} (N + 2s + j) \quad (k\ge 1).
\end{equation}
\end{lemma}
\begin{proof}
The case $k=0$ (and $N$~arbitrary) follows from Theorem~2.1 of~\cite{BG}, together with the value of $\ell_{N,0}$ (for the notation of~\cite{BG} see~(\ref{rad}) and the preceding comments). To go further we pick $k \ge 1$ and apply Lemma~\ref{fn1}:
\begin{align*}
\lim_{r\to\infty}r^{N+2s+k}D^k F_N(r) &= \sum_{k \le 2j \le 2k} (-1)^j \,\alpha_{j,k}\,\lim_{r\to\infty}r^{N+2s+2j}F_{N+2j}(r) \\
&= \sum_{k \le 2j \le 2k} (-1)^j \alpha_{j,k} \, \ell_{N+2j, \, 0}.
\end{align*}
This proves the existence of the finite limit~\eqref{lim} for all~$N,k$. We can recursively determine $\ell_{N,k}$ by applying de l'H\^opital's rule:
\[\ell_{N,k} = \lim_{r \to +\infty} \frac{\, D^k F_N(r) \,}{\, r^{-(N + 2s + k)} \,} = -\lim_{r \to +\infty}\frac{\, r^{N + 2s + k + 1} \, D^{k + 1} F_N(r) \, }{\, N + 2s + k \,} =-\,\frac{\ell_{N, \, k + 1}}{\, N + 2s + k \,},\]
from which~\eqref{ell_Nk} follows.
\end{proof}

\noindent
We are now in a position to prove the estimates~\eqref{dps}-\eqref{dpt} on the derivatives of the fractional heat kernel~$p$.

\begin{proof}[Proof of Proposition~\ref{dp}.]
Let us prove~\ref{dp1}. Denote by $g \in C^\infty(\mathbb R^N)$ the function $g(x) = (2\pi)^\frac{N}2 \, p(x,1)$, so that for all $x \in \mathbb R^N \setminus \{0\}$ we have
\[g(x)=F_N(|x|),\]
and \eqref{rad} rephrases as
\beq\label{pg}
p(x,t)=\frac{t^{-\frac{N}{2s}}}{\, (2\pi)^\frac{N}{2} \,} \, g(t^{-\frac1{2s}} \, x).
\eeq
First we show that for any multi-index $\alpha$, with height $|\alpha|=k$, there exist polynomials $P_{\alpha,1}(\xi), \ldots,P_{\alpha,k}(\xi)$, $\xi \in \R^N$, s.t.\ for all $x\in\R^N\setminus\{0\}$
\beq\label{derg}
D^\alpha g(x)=\sum_{j=1}^k D^jF_N(|x|) \, |x|^{j-k} \, P_{\alpha,j}\Big(\frac{x}{|x|}\Big).
\eeq
Equality~\eqref{derg} is easily proved by induction on $k$. For $k=1$ the multi-index~$\alpha$ coincides with some element~$e_i$,  $i\in\{1,\ldots, N\}$, of the canonical basis of~$\mathbb R^N$, and we have
\[D^\alpha g(x)=DF_N(|x|) \, \frac{x_i}{|x|}.\]
Now fix $k \ge 1$ and assume that \eqref{derg} holds for some $|\alpha|=k$, then let $\alpha'=(\alpha_1,\ldots, \alpha_i+1,\ldots,\alpha_N)$ ($i\in\{1,\ldots,N\}$). By differentiation we find
\begin{align*}
D^{\alpha'} g(x) &= \sum_{j=1}^k\Big[\frac\partial{\partial x_i}\big(D^jF_N(|x|) \, |x|^{j-k}\big) \, P_{\alpha,j}\Big(\frac{x}{|x|}\Big)+D^jF_N(|x|) \, |x|^{j-k} \, \frac\partial{\partial x_i}\, P_{\alpha,j}\Big(\frac{x}{|x|}\Big)\Big] \\
&= \sum_{j=1}^k\Big[D^{j+1}F_N(|x|) \, |x|^{j-k} \, \frac{x_i}{|x|} \,P_{\alpha,j}\Big(\frac{x}{|x|}\Big)+D^jF_N(|x|) \, (j-k) \, |x|^{j-k-1} \,\frac{x_i}{|x|} \, P_{\alpha,j}\Big(\frac{x}{|x|}\Big)\Big] \\
&\qquad+ \sum_{j=1}^k D^jF_N(|x|) \, |x|^{j-k} \, \sum_{h = 1}^N\frac1{|x|} \Big(\delta_{i,h}-\frac{x_i}{|x|} \, \frac{x_h}{|x|}\Big) \, \frac{\partial P_{\alpha,j}}{\partial \xi_h} \Big(\frac{x}{|x|}\Big) \\
&= \sum_{j=1}^{k+1}D^jF_N(|x|) \, |x|^{j-(k+1)} \, P_{\alpha',j}\Big(\frac{x}{|x|}\Big),
\end{align*}
where $\delta_{i,h}$ denotes Kronecker's delta and $P_{\alpha',1},\ldots, P_{\alpha',k+1}$ are convenient polynomials. So \eqref{derg} is achieved.
\vskip2pt
\noindent
Now set
\[M_k=\max\big\{|P_{\alpha,j}(\xi)|:\,|\alpha|=k,\ j=1,\ldots,k,\ \xi\in\partial B_1\big\}.\]
By~\eqref{derg} we have for all $|\alpha|=k$ and all $x\in\R^N\setminus\{0\}$
\[|x|^{N+2s+k} \, |D^\alpha g(x)|\le M_k\sum_{j=1}^k|x|^{N+2s+j} \,\big|D^jF_N(|x|)\big|,\]
and the latter tends to $M_k\sum_{j=1}^k |\ell_{N,j}|$ as $|x|\to\infty$, by Lemma~\ref{fn2}. Consequently the left-hand side is bounded in the complement $B^c_1$ of the unit ball. Furthermore $|D^\alpha g(x)|$ is bounded by continuity in the ball~$B_1$. Hence we can find $C_{N,k}>0$ s.t.\ for all $|\alpha|=k$ and all $x\in\R^N$
\[|D^\alpha g(x)|\le(2\pi)^\frac{N}{2} \, C_{N,k}\,\min\Big\{1,\,\frac1{\, |x|^{N+2s+k} \,}\Big\}\]
(with the convention $\min\{a,\frac1{\, 0 \,}\} = a$). Finally, by~\eqref{pg} we have for all $|\alpha|=k$ and all $(x,t)\in\R^N\times\R^+_0$
\[|D^\alpha p(x,t)| = \frac{\, t^{-\frac{N+k}{2s}} \,}{(2\pi)^\frac{N}{2}} \, |D^\alpha g(t^{-\frac{1}{2s}}x)| \le
C_{N,k}\,\min\Big\{\frac1{\, t^\frac{N + k}{2s} \,},\,\frac{t}{|x|^{N+2s+k}}\Big\},\]
which proves~\ref{dp1}.
\vskip2pt
\noindent
Now we prove~\ref{dp2}. By differentiating~\eqref{scale} with respect to~$x_i$ ($i\in\{1,\ldots N\}$) we get for all $(x,t)\in\R^N\times\R^+_0$
\beq\label{space}
p_i(x,t) = t^\frac{-N-1}{2s} \, p_i(t^{-\frac1{2s}} \, x,1).
\eeq
By differentiating in $t$, instead, we obtain
\[p_t(x,t) = - \, \frac{N}{\, 2s \,} \, t^{-\frac{N+2s}{2s}} \, p(t^{-\frac1{2s}} \, x, \, 1)- \frac1{\, 2s \,} \,  t^{-\frac{N+2s+1}{2s}}\,\sum_{i = 1}^N x_i \, p_i(t^{-\frac1{2s}} \, x, \, 1).\]
Using~\eqref{scale} again and recalling~\eqref{space}, we may write
\[p_t(x,t) = -\,\frac{N}{\, 2s \,} \, t^{-1} \, p(x,t) - \frac1{\, 2s \,} \,  t^{-1}\,\sum_{i = 1}^N x_i \, p_i(x,t).\]
Letting $|\alpha| = 1$ in~\eqref{dps} and using~\eqref{ke0} we arrive at
\beq\label{twobounds}
|p_t(x,t)| \le C_1\,\min\Big\{\frac1{\, t^\frac{N+2s}{2s} \,}, \, \frac1{\, |x|^{N + 2s} \,} \Big\} + C_2\,\min\Big\{\frac{|x|}{\, t^\frac{N+2s+1}{2s} \,}, \, \frac1{\, |x|^{N + 2s} \,} \Big\},
\eeq
for convenient $C_1,C_2 > 0$. In fact, the second term in the estimate above is dominated by the first, as two cases may occur:
\begin{itemize}[leftmargin=0.5cm]
\item if $|x|<t^\frac{1}{2s}$, then $t^{-\frac{N+2s+1}{2s}} \, |x| < t^{-\frac{N+2s}{2s}}$, hence
\[\min\Big\{\frac{|x|}{\, t^\frac{N+2s+1}{2s} \,}, \, \frac1{\, |x|^{N + 2s} \,} \Big\} = \frac{|x|}{\, t^\frac{N+2s+1}{2s} \,} < \frac1{\, t^\frac{N+2s}{2s} \,} = \min\Big\{\frac1{\, t^\frac{N+2s}{2s} \,}, \, \frac1{\, |x|^{N + 2s} \,} \Big\};\]
\item if $|x|\ge t^\frac{1}{2s}$, then we easily get
\[\min\Big\{\frac{|x|}{\, t^\frac{N+2s+1}{2s} \,}, \, \frac1{\, |x|^{N + 2s} \,} \Big\} = \frac{1}{\,|x|^{N+2s}\,} = \min\Big\{\frac1{\, t^\frac{N+2s}{2s} \,}, \, \frac1{\, |x|^{N + 2s} \,} \Big\}.\]
\end{itemize}
Therefore \eqref{twobounds} proves \ref{dp2}.
\end{proof}

\begin{remark}
Note that the derivative $D^\alpha p(x,t)$ may well vanish somewhere, hence $|D^\alpha p(x,t)|$ cannot be estimated from below for all $|\alpha| = k > 0$ by means of a positive function over $\mathbb R^N \times \R^+_0$ (as for $p(x,t)$ in~\eqref{ke0}). As an example we may let $\alpha = (1, 0, \dots, 0)$. Since $p(\cdot,t)$ is a radial function we have $D^\alpha p(x,t) = 0$ whenever $x_1 = 0$. In fact, all directional derivatives of the first order along a tangential (i.e., orthogonal to~$x$) direction vanish identically. As a further example we may let $\alpha = (2, 0, \dots, 0)$ and $x = (r, 0, \dots, 0)$, $r > 0$. In this case we have $D^\alpha g(x) =D^2 F_N(r)$, which is positive for $r$ large by Lemma~\ref{fn2}. Besides, $D^2 F_N(r)$ is negative for $r$ close to~$0$ by Lemma~\ref{fn1} and by~(\ref{rad}), therefore $D^\alpha g(x)$ must vanish somewhere.
\end{remark}

\medskip

\noindent
{\small {\bf Acknowledgement.} The authors are members of the Gruppo Nazionale per l'Analisi Matematica, la Probabilit\`a e le loro Applicazioni (GNAMPA) of the Istituto Nazionale di Alta Matematica (INdAM). This work was partially supported by the GNAMPA Research Project 'Regolarit\`a, esistenza e propriet\`a geometriche per le soluzioni di equazioni con operatori frazionari non lineari' (2017). We would like to thank Prof.\ Filomena Pacella for some fruitful discussion, which allowed us to enrich the results of this paper. We are grateful to the anonymous Referee for having carefully read the manuscript and having pointed out some relevant references.}


\medskip

\begin{thebibliography}{99}

\bibitem{BPSV}
{\sc B.\ Barrios, I.\ Peral, F.\ Soria, E.\ Valdinoci,}
A Widder's type theorem for the heat equation with non-local diffusion,
{\em Arch. Rational Mech. Anal.} {\bf 213} (2014) 629--650.

\bibitem{BG}
{\sc R.M.\ Blumenthal, R.H.\ Getoor,}
Some theorems on stable processes,
{\em Trans. Amer. Math. Soc.} {\bf 95} (1960) 263--273.
 
\bibitem{BGR}
{\sc K.\ Bogdan, T.\ Grzywny, M.\ Ryznar,}
Heat kernel estimates for the fractional Laplacian with Dirichlet conditions,
{\em Ann. Prob.} \textbf{38} (2010) 1901--1923.

\bibitem{BV}
{\sc C.\ Bucur, E.\ Valdinoci,}
Non-local diffusion and applications,
Springer, New York (2016).

\bibitem{CS}
{\sc X.\ Cabr\'e, Y.\ Sire,}
Nonlinear equations for fractional Laplacians I: Regularity, maximum principles, and Hamiltonian estimates,
{\em Ann. Inst. H. Poincar\'e Anal. Non Lin\'eaire} \textbf{367} (2014) 911--941.

\bibitem{CS1}
{\sc X.\ Cabr\'e, Y.\ Sire,}
Nonlinear equations for fractional Laplacians II: Existence, uniqueness, and qualitative properties of solutions,
{\em Trans. Amer. Math. Soc.} \textbf{31} (2014) 23--53.

\bibitem{C}
{\sc L.\ Caffarelli},
Non-local diffusions, drifts and games,
in {\sc H.\ Holden, K.H.\ Karlsen} (eds.), Nonlinear partial differential equations, Springer, New York (2012).

\bibitem{CF}
{\sc L.\ Caffarelli, A.\ Figalli},
Regularity of solutions to the parabolic fractional obstacle problem,
{\em J. Reine Angew. Math.} {\bf 680} (2013) 191--233.

\bibitem{CS2}
{\sc L.\ Caffarelli, L.\ Silvestre},
An extension problem related to the fractional Laplacian, 
{\em Comm. Partial Differential Equations} {\bf 32} (2007) 1245--1260.

\bibitem{CKS}
{\sc Z.Q.\ Chen, P.\ Kim, R.\ Song,}
Heat kernel estimates for the Dirichlet fractional Laplacian,
{\em J. Eur. Math. Soc.} \textbf{12} (2010) 1307--1329.

\bibitem{DPV}
{\sc E.\ Di Nezza, G.\ Palatucci, E.\ Valdinoci},
Hitchhiker's guide to the fractional Sobolev spaces,
{\em Bull. Sci. Math.} {\bf 136} (2012) 521--573.

\bibitem{Evans}
{\sc L.\ C.\ Evans,}
Partial differential equations, 2nd edition,
{\em Graduate Studies in Mathematics} {\bf 19}, American Mathematical Society, Providence, Rhode Island, 2010.

\bibitem{FR}
{\sc X.\ Fern\'andez-Real, X.\ Ros-Oton},
Boundary regularity for the fractional heat equation,
{\em Rev. R. Acad. Cienc. Exactas F\'{\i}s. Nat. Ser. A Math.} {\bf 110} (2016) 49--64.

\bibitem{G}
{\sc A.\ Greco},
Convex functions over the whole space locally satisfying fractional equations,
{\em Minimax Theory Appl.}, to appear.

\bibitem{GS}
{\sc A.\ Greco, R.\ Servadei},
Hopf's lemma and constrained radial symmetry for the fractional Laplacian,
{\em Math. Res. Lett.} \textbf{23} (2016), 863--885.

\bibitem{IMS}
{\sc A.\ Iannizzotto, S.\ Mosconi, M.\ Squassina},
$H^s$ versus $C^0$-weighted minimizers,
{\em Nonlinear Differ. Equ. Appl.} {\bf 22} (2015) 477--497.

\bibitem{IKM}
{\sc K.\ Ishige, T.\ Kawakami H.\ Michihisa},
Asymptotic expansions of solutions of fractional diffusion equations,
\href{https://arxiv.org/abs/1610.09789v1}{arXiv:1610.09789}.

\bibitem{IS}
{\sc K.\ Ishige, P.\ Salani},
A note on parabolic power concavity,
{\em Kodai Math. J.} {\bf 37} (2014) 668--679.

\bibitem{J}
{\sc F.\ John,}
Partial differential equations (fourth edition),
Springer, New York (1982).

\bibitem{Kulczycki}
{\sc T. Kulczycki,}
On concavity of solution of Dirichlet problem for the
equation $(-\Delta)^{1/2 \,} \varphi = 1$ in a convex planar region,
{\em J. Eur. Math. Soc. (JEMS)} \textbf{19} (2017) 1361--1420.

\bibitem{KR}
{\sc T.\ Kulczycki, M.\ Ryznar,}
Gradient estimates of harmonic functions and transition densities for L\'evy processes,
{\em Trans. Amer. Math. Soc.} {\bf 368} (2016) 281--318. 

\bibitem{MN}
{\sc R.\ Musina, A.\ Nazarov},
On fractional Laplacians,
{\em Comm. Partial Differential Equations} {\bf 39} (2014) 1780--1790. 

\bibitem{P}
{\sc G. P\'olya,}
On the zeros of an integral function represented by Fourier's integral,
{\em J. London Math. Soc.} {\bf 1} (1926) 98--99.

\bibitem{Protter-Weinberger}
{\sc M.\ H.\ Protter, H.\ F.\ Weinberger,}
Maximum principles in differential equations,
Springer-Verlag, New York, 1984.

\bibitem{RS}
{\sc X.\ Ros-Oton, J.\ Serra},
The Dirichlet problem for the fractional Laplacian: regularity up to the boundary,
{\em J. Math. Pures Appl.} {\bf 101} (2014) 275-302.

\bibitem{SV}
{\sc R.\ Servadei, E.\ Valdinoci},
Mountain pass solutions for non-local elliptic operators,
{\em J. Math. Anal. Appl.} {\bf 389} (2012) 887--898.

\bibitem{SV1}
{\sc R.\ Servadei, E.\ Valdinoci},
On the spectrum of two different fractional operators,
{\em Proc. Roy. Soc. Edinburgh Sec. A} {\bf 144} (2014) 831--855.

\bibitem{VTV}
{\sc L.\ V\'azquez, J.J.\ Trujillo, M.P.\ Velasco},
Fractional heat equation and the second law of thermodynamics,
{\em Fract. Calc. Appl. Anal.} {\bf 14} (2011) 334--342.

\bibitem{W}
{\sc D.V.\ Widder,}
Positive temperatures on an infinite rod,
{\em Trans. Amer. Math. Soc.} {\bf 55} (1944) 85--95.

\bibitem{mediawiki}
Fractional heat equation: \href{https://www.ma.utexas.edu/mediawiki/index.php/Fractional_heat_equation}{https://www.ma.utexas.edu/mediawiki/index.php/Fractional\_heat\_equation}

\end{thebibliography}
\end{document}